\documentclass[11pt]{article}

\usepackage{pdfsync}

\usepackage{latexsym,amssymb,amsfonts,amsmath,epsfig,tabularx,stmaryrd}

\usepackage[hidelinks=true]{hyperref}

\usepackage{soul}
\usepackage{cite}
\usepackage[normalem]{ulem}

\usepackage[usenames]{xcolor}
\definecolor{darkred}{RGB}{139,0,0}
\definecolor{darkgreen}{RGB}{0,100,0}
\definecolor{darkmagenta}{RGB}{139,0,139}
\definecolor{darkorange}{RGB}{220,110,20}
\newcommand{\is}[1]{{\color{blue}{#1}}}

\newtheorem{theorem}{Theorem}
\newtheorem{remark}[theorem]{Remark}
\newtheorem{lemma}[theorem]{Lemma}

\newenvironment{proof}{\begin{trivlist}\item[\hskip\labelsep{\bf Proof.}]}{$\hfill\Box$\end{trivlist}}
\newenvironment{proofof}[1]{\begin{trivlist}
    \item[\hskip\labelsep{\bf Proof of {#1}.}]}{$\hfill\Box$\end{trivlist}}

\newcommand{\esup}{\operatornamewithlimits{ess\,sup}}

\newcommand{\bx}{{\boldsymbol{x}}}

\newcommand{\bu}{{\boldsymbol{u}}}

\newcommand{\bw}{{\boldsymbol{w}}}

\newcommand{\bsone}{{\boldsymbol{1}}}
\newcommand{\bsalpha}{{\boldsymbol{\alpha}}}

\newcommand{\bszeta}{{\boldsymbol{\zeta}}}

\newcommand{\setu}{{\mathfrak{u}}}
\newcommand{\setv}{{\mathfrak{v}}}
\newcommand{\setD}{{\mathfrak{D}}}

\newcommand{\satop}[2]{\stackrel{\scriptstyle{#1}}{\scriptstyle{#2}}}

\newcommand{\R}{\mathbb{R}}
\newcommand{\N}{\mathbb{N}}
\newcommand{\rd}{{\mathrm{d}}}

\newcommand{\calC}{{\mathcal{C}}}

\newcommand{\calL}{{\mathcal{L}}}

\newcommand{\calW}{{\mathcal{W}}}

\newcommand{\bse}{{\boldsymbol{e}}}

\newcommand{\bsx}{{\boldsymbol{x}}}
\newcommand{\bsy}{{\boldsymbol{y}}}

\newcommand{\bsr}{{\boldsymbol{r}}}

\newcommand{\bsgamma}{{\boldsymbol{\gamma}}}

\newcommand{\bseta}{{\boldsymbol{\eta}}}
\newcommand{\bmix}{{\mathrm{mix}}}

\newcommand{\bbR}{{\mathbb{R}}}

\newcommand{\mask}[1]{}
\newcommand{\ind}{\mathop{\rm ind}}
\newcommand{\mix}{\mathrm{mix}}

\title{High dimensional integration of kinks and jumps \\
--  smoothing by preintegration}

\author{Andreas Griewank, Frances Y.~Kuo, Hernan Le\"ovey and Ian H.~Sloan}

\date{August 2017}

\begin{document}

\maketitle

\begin{abstract}
We show how simple kinks and jumps of otherwise smooth integrands over
$\R^d$ can be dealt with by a preliminary integration with respect to a
single well chosen variable. It is assumed that this preintegration, or
conditional sampling,  can be carried out with negligible error, which is
the case in particular for option pricing problems. It is proven that
under appropriate conditions the preintegrated function of $d-1$ variables
belongs to appropriate mixed Sobolev spaces, so potentially allowing high
efficiency of Quasi Monte Carlo and Sparse Grid Methods applied to the
preintegrated problem. The efficiency of applying Quasi Monte Carlo to the
preintegrated function are demonstrated on a digital Asian option using
the Principal Component Analysis factorisation of the covariance matrix.

\end{abstract}

\section{Introduction} \label{sec:intro}
 \vspace{.2cm}

In the present paper we analyse a natural method for numerical integration
over $\R^d$, where $d$ may be large, in the presence of ``kinks'' (i.e.\
discontinuities in the gradients) or ``jumps'' (i.e.\ discontinuities in
the function).  In this method one of the variables is integrated out in a
``preintegration'' step, with the aim of creating a smooth integrand over
$\R^{d-1}$.

Integrands with kinks and jumps arise in option pricing, because an option
is normally considered worthless if the value falls below a predetermined
strike price.  In the case of a continuous payoff function this introduces
a kink, while in the case of a binary or other digital option it
introduces a jump.

A simple strategy is to ignore the kinks and jumps, and apply directly a
method for integration over $\R^d$. While there has been very significant
recent progress in \emph{Quasi Monte Carlo \textnormal{(}QMC\textnormal{)}
methods} \cite{DKS13} and \emph{Sparse Grid \textnormal{(}SG\textnormal{)}
methods} \cite{BG04} for high dimensional integration when the integrand
is somewhat smooth, there has been little progress in understanding their
performance when the integrand has kinks or jumps.

The performance of QMC and SG methods is degraded in the presence of kinks
and jumps, but perhaps not as much as might have been expected, given that
in both cases the standard error analysis fails in general for kinks and
jumps: the standard assumption in both cases is that the integrand has
mixed first partial derivatives for all variables, or at least that it has
bounded Hardy and Krause variation over the unit cube $[0,1]^d$, whereas
even a straight non-aligned kink (one that is not orthogonal to one of the
axes) lacks mixed first partial derivatives even for $d=2$, and generally
exhibits unbounded Hardy and Krause variation on $[0,1]^d$ for $d\ge3$
\cite{Owen05}.

A possible path towards understanding the performance of QMC and SG
methods in the presence of kinks and jumps was developed in \cite{GKS13}.
That paper studied the terms of the ``ANOVA decomposition'' of functions
with kinks defined on $d$-dimensional Euclidean space $\bbR^d$, and showed
that under suitable circumstances all but one of the $2^d$ ANOVA terms can
be smooth, with the single exception of the highest order ANOVA term, the
one depending on all $d$ of the variables. If the ``effective dimension''
of the function is small, as is commonly thought to be the case in
applications, then that single non-smooth term can be expected to make a
very small contribution to both supremum and $\calL_2$ norms. In a
subsequent paper \cite{GKSnote} the same authors showed, by strengthening
the theorems and correcting a mis-statement in \cite{GKS13}, that the
smoothing of all but the highest order ANOVA term is a reality for the
case of an arithmetic Asian option with the Brownian bridge construction.

More precisely, the papers \cite{GKS13} and \cite{GKSnote} showed, for a
function of the form  $f(\bsx)= \max(\phi(\bsx),0)$ with $\phi$ smooth (so
that $f$ generically has a kink along the manifold $\phi(\bsx)=0$), that
if the $d$-dimensional function $\phi$ has a positive partial derivative
with respect to $x_j$ for some $j\in\{1,\ldots,d\}$, and if certain growth
conditions at infinity are satisfied, then all the ANOVA terms of $f$ that
do not depend on the variable $x_j$ are smooth. The underlying reason, as
explained in \cite{GKS13}, is that integration of $f$ with respect to
$x_j$, under the stated conditions, results in a $(d-1)$-dimensional
function that no longer has a kink, and indeed is as often differentiable
as the function~$\phi$.

Going beyond kinks, we prove in this paper that Theorem 1 in
\cite{GKSnote} can be extended from kinks to jumps -- thus jumps are
smoothed under almost the same conditions as kinks. The smoothing occurs
even in situations (such as occur in option pricing) where the location of
the kink or jump treated as a function of the other $d-1$ variables moves
off to infinity for some values of the other variables.

In this paper we pay particular attention to proving that the presmoothed
integrand belongs to an appropriate mixed-derivative function space.

The preintegration method studied in the present paper has appeared as a
practical tool under other names in many other papers, including those
related to ``conditional sampling'' (see \cite{GlaSta01}; the paragraph
leading up to and including Lemma~7.2 in \cite{ACN13a}; the remark at the
end of Section~3 in \cite{ACN13b}), and other root-finding strategies for
identifying where the payoff is positive (see \cite{Hol11,NuyWat12}), as
well as those under the name ``smoothing'' (see \cite{BST17,WWH17}). In
contrast to the cited papers, the emphasis in this paper is on rigorous
analysis. Also, we here prefer the description ``preintegration'' because
to us ``conditional sampling'' suggests a probabilistic setting, which is
not necessarily relevant here.

Even for the classical \emph{Monte Carlo \textnormal{(}MC\textnormal{)}
method} the preintegration step can be useful: to the extent that the
preintegration can be considered exact, there is a reduction in the
variance of the integrand, by the sum of the variances of all ANOVA terms
that involve the preintegration variable $x_j$ (since the ANOVA terms are
eliminated because their exact integrals with respect to $x_j$ are all
zero). In our numerical experiments that reduction proves to be quite
significant.

The problem class and the method are stated in Section~\ref{sec:method}.
Immediately Section~\ref{sec:app} gives numerical examples in the context
of an option pricing problem with 256 time steps, treated as a problem of
integration in 256 dimensions. Section~\ref{sec:var} briefly discusses the
variance reduction by preintegration for $\calL_2$ functions.
Section~\ref{sec:smooth} focuses on the smoothing effect of
preintegration. It gives mathematical background on needed function spaces
and states two new smoothing theorems, extended here in a non-trivial way
from \cite[Theorem~1]{GKSnote}. Section~\ref{sec:apply} applies our
theoretical results to the option pricing example. Technical proofs are
given in Section~\ref{sec:proof1}.

\section{The problem and the method} \label{sec:method}

The problem is the approximate evaluation of
\begin{equation} \label{problem1}
 I_d f
 \,:=\, \int_{\R^d}f(\bsx)\rho_d(\bsx)\,\rd\bsx
 \,=\, \int_{-\infty}^\infty\ldots\int_{-\infty}^\infty
       f(x_1,\ldots,x_d)\,\rho_d(\bsx)\,\rd x_1 \cdots\rd x_d,
\end{equation}
with
\[
  \rho_d(\bsx) \,:=\, \prod_{k=1}^d \rho(x_k),
\]
where $\rho$ is a continuous and strictly positive probability density
function on $\R$ with some smoothness, and $f$ is a real-valued function
integrable with respect to $\rho_d$.

To allow for both kinks and jumps we assume that the integrand is of the
form
\begin{equation} \label{problem2}
 f(\bsx) \,=\, \theta(\bsx)\,\ind(\phi(\bsx)),
\end{equation}
where $\theta$ and $\phi$ are somewhat smooth functions, and $\ind(\cdot)$
is the indicator function which gives the value $1$ if the input is
positive and $0$ otherwise. When $\theta = \phi$ we have $f(\bsx) =
\max(\phi(\bsx),0)$ and thus we have the familiar kink seen in option
pricing through the occurrence of a strike price. When $\theta$ and $\phi$
are different (for example, when $\theta(\bsx) = 1$) we have a structure
that includes binary digital options.

Our key assumption on $\phi(\bsx)$ is that it has a positive partial
derivative (and so is an increasing function) with respect to some
variable $x_j$, that is, we assume that for some $j\in\{1,\ldots,d\}$ we
have
\begin{equation} \label{monotone}
 \frac{\partial \phi}{\partial x_j}(\bsx) > 0
 \qquad \mbox{for all} \quad
 \bsx \in \R^d.
\end{equation}
In other words $\phi$ is monotone with respect to $x_j$.

We also make an assumption about the behavior as $x_j\to +\infty$.  To
state this it is convenient, given $j\in\{1,\ldots,d\}$, to write the
general point $\bsx\in\R^d$ as $\bsx= (x_j,\bsx_{-j})$, where $\bsx_{-j}$
denotes the vector of length $d-1$ denoting all the variables other than
$x_j$. With this notation, a second assumption is that
\begin{equation} \label{unbounded}
  \lim_{x_j\to +\infty } \phi(\bsx)
  \,=\, \lim_{x_j \to +\infty}\phi(x_j,\bsx_{-j})
  \,=\, +\infty
  \qquad\mbox{for fixed}\quad \bsx_{-j}\in \bbR^{d-1}.
\end{equation}
The latter growth property follows automatically if we assume in addition to
\eqref{monotone} that $(\partial^2 \phi/ \partial x_j^2)(\bsx)\ge0$ for
all $\bsx\in \bbR^d$. Additional properties at infinity will be assumed in
Theorems~\ref{thm:main1} and~\ref{thm:main2}.

Assuming that the properties \eqref{monotone} and \eqref{unbounded} both
hold for some $j\in\{1,\ldots,d\}$, the method is easily described: we
write \eqref{problem1} as the repeated integral using Fubini's theorem
\[
  I_d f
  \,=\, \int_{\R^{d-1}}\left(\int_{-\infty}^\infty f(x_j,\bsx_{-j})\,\rho(x_j)\,\rd x_j \right)
  \rho_{d-1}(\bsx_{-j})\,\rd\bsx_{-j},
\]
and first evaluate  the inner integral for each needed value of
$\bsx_{-j}$. This is the ``\emph{preintegration}'' step. The essential
point of the method is that the outer integral can then be evaluated by a
standard QMC or SG method, in the knowledge that the integrand for this
$(d-1)$-dimensional integral is smooth.

Looking more closely at the preintegration step, we write the operation of
integration with respect to $x_j$ as
\begin{equation}\label{firstPj}
  (P_j f)(\bsx_{-j})
  \,:=\, \int_{-\infty}^\infty f(x_j,\bsx_{-j})\,\rho(x_j)\,\rd x_j.
\end{equation}
It follows from \eqref{problem2} and \eqref{monotone} that the integrand
in this integral has generically a jump at the (unique) point at which
$\phi(x_j, \bsx_{-j})$ passes through zero.  By the implicit function
theorem (see Theorem~\ref{thm:implicit} below) for each $\bsx_{-j}$ there
is a unique value $\psi(\bsx_{-j})$ of $x_j$ at which
$\phi(x_j,\bsx_{-j})$ passes from negative to positive values with
increasing $x_j$. The preintegration step may then be written as
\[
  (P_j f)(\bsx_{-j})
  \,=\,
  \int_{\psi(\bsx_{-j})}^\infty f(x_j,\bsx_{-j})\,\rho(x_j )\,\rd x_j.
\]
An essential ingredient in any implementation of the method is the
accurate evaluation of $\psi(\bsx_{-j})$, for each point $\bsx_{-j}$ of
the outer integration rule. The semi-infinite integral $P_jf$ may then be
evaluated, for each needed point $\bsx_{-j}$, by a standard method for
1-dimensional integrals, for example by a formula of Gauss type.  On the
other hand, in certain important applications such as option pricing, the
integration can be performed in more or less closed form.

The monotonicity condition \eqref{monotone} and the infinite growth
condition \eqref{unbounded} imply that for fixed $\bsx_{-j}$ the function
$\phi(x_j, \bsx_{-j})$ either has a simple root $x_j=\psi(\bsx_{-j})$ or
is positive for all $x_j\in\R$. The zero set of $\phi$, denoted by
$$\phi^{-1}(0) \,:=\, \{\bsx\in\R^d:\phi(\bsx)=0\},$$
is then a hypersurface, i.e., a continuous manifold of dimension $d-1$.
However, its projection onto $\R^{d-1}$ obtained by ignoring the component
$x_j$ can be very complicated, even if $\phi$ is highly differentiable.

An example with $d=2$  and $j=1$ illustrating the complications that can
arise is given by
\begin{align}\label{eq:example}
\phi(x_1,x_2)  \,:=\, \begin{cases}
\exp(x_1) - x_2^m \sin(1/x_2) & \mbox{for }\, x_2>0,\\
\exp(x_1) & \mbox{for }\,  x_2\le 0, \end{cases}
\end{align}
for some large $m$. Since $\phi(x_1,x_2)$ is monotonically increasing in
$x_1$, the explicit solution of $\phi(x_1,x_2)=0$ is
$$  x_1 \,= \,   \psi(x_2) \,:=\, m \log(x_2)
+ \log((\sin(1/x_2))_+) \quad \mbox{for }\, x_2 \in U_1,$$ where $z_+ :=
\max(0,z)$, and
\begin{align*}
 U_1 \,:=\,
 \{x_2\in \R \;:\; \phi(x_1,x_2)=0 \;\mbox{ for some }\,x_1\in\R\}
 \,=\, \left  (\tfrac{1}{\pi}, \infty  \right ) \,\cup\,
 \bigcup_{k \in \N}  \left ( \tfrac{1}{ (2 k+1)\pi},  \tfrac{1}{ 2 k\pi} \right ),
\end{align*}
while $\phi(x_1,x_2)=0$ has no solution for $x_2$ in the complicated
complementary set
\[
  U_1^+ \,:=\,
  (-\infty, 0] \,\cup\, \bigcup_{k \in \N}  \left [ \tfrac{1}{ 2 k\pi}, \tfrac{1}{(2 k- 1)\pi} \right ].
\]

\begin{figure}[h]
\vspace*{-3cm}
\begin{center}
\includegraphics[width=12cm,height=14cm]{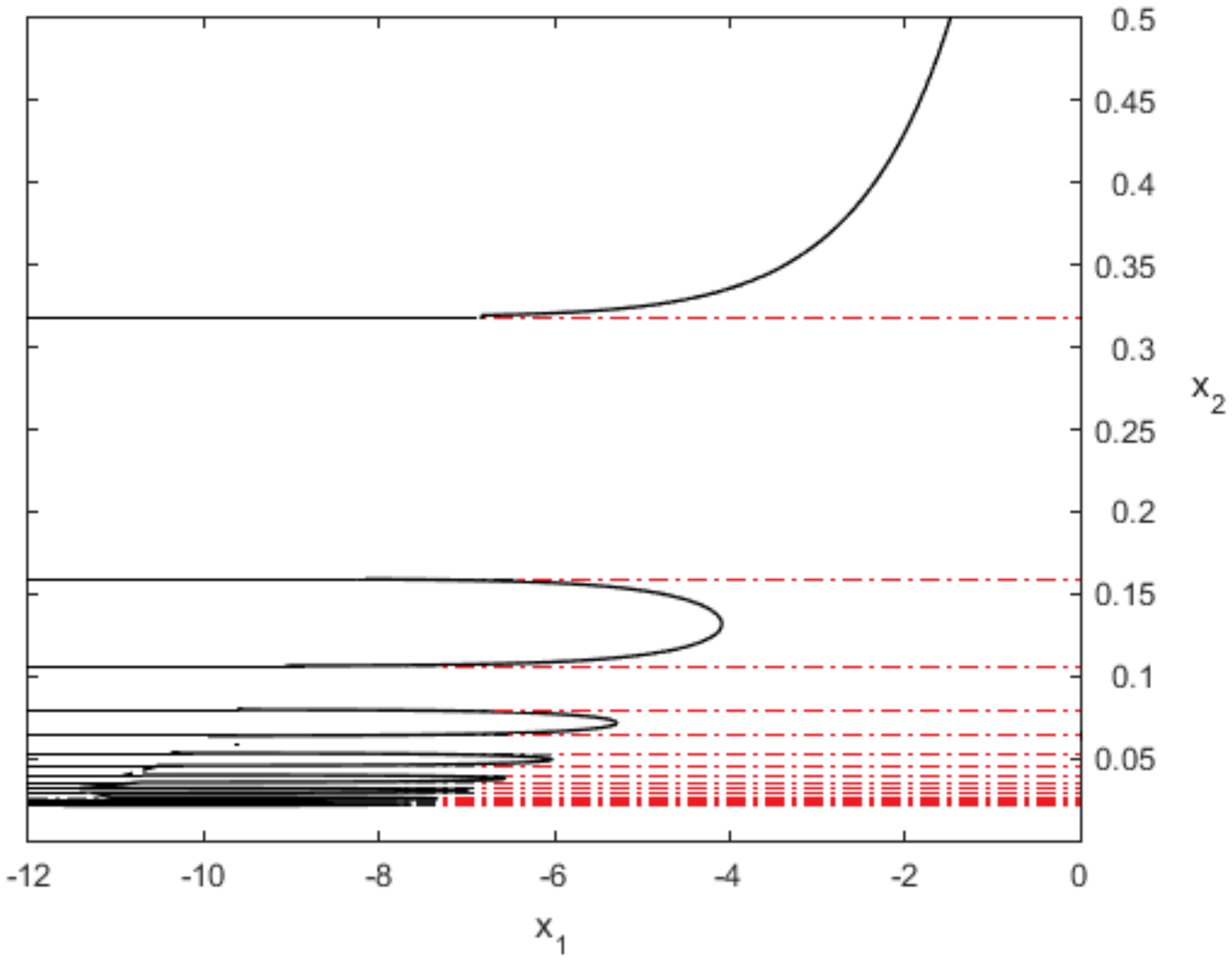}
\end{center}
\vspace*{-4.3cm}
\caption{The zero set $\phi^{-1}(0)$ for the example in
\eqref{eq:example} with $m=2$.} \label{figure:example}
\end{figure}

The graph of the zero set $\phi^{-1}(0)$ for $m=2$ is shown in Figure
\ref{figure:example}. The solid lines represent the zero set, while the
broken lines parallel to the horizontal axis define the boundaries on the
$x_2$ axis between subsets of $U_1^+$ for which there is no solution of
$\phi(x_1,x_2) = 0$ and subsets of $U_1$  for which there is a solution
$x_1 = \psi(x_2)$. The preintegrated version of $f$ given by
\eqref{problem2} for any smooth $\theta$ will rather clearly be
differentiable on both $U_1$ and $U_1^+$, but it is not obvious that this
is the case on the complicated boundary between the two sets. To ensure
the necessary differentiability properties it turns out in Section
\ref{sec:smooth} to be necessary to assume that the functions $\theta$ and
$\phi$ and their derivatives, when multiplied by the appropriate weight
functions, decay sufficiently rapidly as $x_1= \psi(x_2)$ runs to
$-\infty$.


\section{Numerical experiments: application to option pricing} \label{sec:app}

In this section we apply the preintegration method to an option pricing
example, for which the payoff function is discontinuous.

An important aspect of the method presented in this paper is that the user
needs to choose a variable $x_j$ such that the condition \eqref{monotone}
is satisfied. In the paper \cite{GKS13} it is shown that for the standard
and Brownian bridge constructions for path simulation of Brownian motions
every choice of the variable $x_j$ will be suitable.  More interesting for
the present paper is the popular Principal Component Analysis (or PCA)
method of constructing the Brownian motion \cite{Glasserman}: for this
case the only result known to us, from \cite[Section 5]{GKS10}, is that
the property \eqref{monotone} is guaranteed if $x_j$ is the variable that
corresponds to the largest eigenvalue of the Brownian motion covariance
matrix.  For this reason it is of particular interest to apply the present
theory to the PCA case, as we do below.

For our tests, we consider now the example of an arithmetic average
digital Asian option. We assume that the underlying asset $S_t$ follows
the \emph{geometric Brownian motion} model based on the stochastic
differential equation
\begin{equation}\label{eq:brown:sode}
 \rd S_t \,=\, r S_t \,\rd t + \sigma S_t \,\rd W_t\,,
\end{equation}
where $r$ is the risk-free interest rate, $\sigma$ is the (constant)
volatility and $W_t$ is the standard Brownian Motion. The solution of this
stochastic equation can be given as
\begin{equation}\label{eq:brown:solution}
 S_t \,=\, S_0 \exp \left ( \left ( r - \tfrac{\sigma^2}{2}\right ) t + \sigma W_t \right).
 \end{equation}
The problem of simulating asset prices can be reduced to the problem of
simulating discretized Brownian motion paths taking values
\mbox{$W_{t_1},\dots,W_{t_d}$}, where $d$ is the number of time steps
taken in the disctretization of the continuous time period $[0, T]$. In
our tests, the asset prices are assumed to be sampled at equally spaced
times $t_\ell:=\ell \Delta t$, $\ell=1,\dots,d$, where $\Delta t := T/d $.
The Brownian motion is a Gaussian process, therefore the vector
$(W_{t_1},\dots,W_{t_d})$ is normally distributed, and in this particular
case is a vector with mean zero and covariance matrix $C$ given by
\[
C \,=\, (\min( t_\ell,t_k ))_{\ell,k=1}^d .
\]

The value of an arithmetic average digital Asian call option is
\[ 
V \,=\, \frac{e^{-rT}}{(2\pi)^{d/2} \sqrt{\det(C)}}
\int_{\R^d}
{\rm ind}\bigg( \frac{1}{d}\,\sum_{\ell=1}^{d} S_{t_\ell}(w_\ell) - K\bigg)
e^{- \frac{1}{2} \bw^\top C^{-1} \bw} \,\rd\bw \,,
\] 
with $\bw=(W_{t_1},\dots,W_{t_d})^\top$.
After a factorization $C=AA^\top$ of the covariance matrix is chosen (for
the choice of $A$ is not unique), we can rewrite the integration problem
using the substitution $\bw=A\bx$ as
\[ 
V \,=\, \frac{e^{-rT}}{(2\pi)^{d/2}}
\int_{\R^d}
{\rm ind}\bigg( \frac{1}{d}\,\sum_{\ell=1}^{d} S_{t_\ell}((A\bx)_\ell) - K\bigg)
e^{- \frac{1}{2} \bx^\top \bx} \,\rd\bx \,.
\] 
The new variable vector $\bx=(x_1,\ldots,x_d)^{\top}$ can be assumed to
consist of independent standard normally distributed random variables.
Then the identity $\bw=A\bx$ defines a construction method for Brownian
paths. We therefore have an integral of the form
\eqref{problem1}--\eqref{problem2} with $\rho(x) = e^{-x^2}/\sqrt{2\pi}$,
$\theta(\bsx) = e^{-rT}$, and
\begin{equation}\label{phi_explicit}
  \phi(\bsx) \,=\,
  \frac{S_0}{d} \sum_{\ell=1}^d
  \exp\bigg( \left(r-\tfrac{\sigma^2}{2}\right) \ell\Delta t
  +\sigma \sum_{k=1}^d A_{\ell k}\, x_k\bigg) -K.
\end{equation}

We use in our experiments the PCA factorization of $C$, which is based on
the orthogonal factorization
\[
C \,=\, (\bu_{1};\ldots;\bu_{d})\,
{\rm diag}(\lambda_{1},\ldots,\lambda_{d})\,(\bu_{1};\ldots;\bu_{d})^{\top},
\]
where the eigenvalues $\lambda_{1},\ldots,\lambda_{d}$ (all positive) are
given in non-increasing order, with corresponding unit-length column
eigenvectors $\bu_1\ldots,\bu_d$, and as a result
\[
A \,=\, (\sqrt{\lambda_{1}}\bu_{1};\ldots;\sqrt{\lambda_{d}}\bu_{d}).
\]
Note that we have $A_{\ell 1}>0$ for $1\le \ell \le d$ because the
elements of the eigenvector $\bu_1$ 
are all positive.

For approximate integration with quadrature, we generate randomized QMC or
MC samples $\bx^{(1)},\dots,\bx^{(N)}$ over $\R^d$ by first generating
classical randomized QMC or MC samples over unit cube $(0,1)^d$, and then
transforming them to $\R^d$ using in each coordinate the univariate
inverse normal cumulative distribution function $\Phi^{-1}({\cdot})$.
The randomized QMC points over $(0,1)^d$ are obtained by first generating
Sobol$'$ points over $[0,1]^d$ with direction numbers taken from
\cite{JoeKuo08}, and then applying the random linear-affine scrambling
method as proposed by Matousek \cite{Matousek} (as implemented in the
statistics toolbox of MATLAB). Note that taking randomly scrambled
Sobol$'$ points not only allows us to generate statistically independent
QMC samples, but also allows us to avoid in practice having points lying
on the boundary of $(0,1)^d$ (which is usually the case for non-randomized
QMC points), since the boundary is sampled with zero probability. The MC
points were taken from the Mersenne Twister PRNG. For the function
$\Phi^{-1}({\cdot})$, we have used Moro's algorithm~\cite{Glasserman}. The
matrix $A$ can be given explicitly \cite{GKS10}, but more importantly,
each matrix-vector multiplication $A\bx^{(i)}, 1\le i \le N$ can be done
with $O(d\,\log d)$ cost by means of the fast-sine transform \cite{Sche07}
(as long as time steps for discretization are taken of equal size).

For the preintegration approach, we generate randomized QMC or MC points
over $[0,1]^{d-1}$, following the procedure for the $d$-dimensional case,
and so obtain $N$ sample points over $\R^{d-1}$. We then evaluate the
paths without using the first variable $x_1$, i.e., we sample over the
coordinates $x_2,\dots,x_d$. Once a sample point on these coordinates is
fixed, the resulting problem is a one-dimensional integral on the variable
$x_1$. We take then the approximation
\[
 V \,\approx\, Q_{N,d-1} \left( P_1(f)\right)
 \,=\, P_1\left( Q_{N,d-1}(f) \right)
 \,=\, \frac{1}{N} \sum_{i=1}^N \int_{-\infty}^\infty f(\xi,\bx^{(i)})\,\rho(\xi) \,\rd \xi,
\]
where the quadrature with respect to $\xi$ is to be carried out for each
of the $N$ sample points $\bx^{(i)}$ in $\R^{d-1}$. In the PCA case in
this problem, the resulting $N$ univariate integrals can be calculated in
terms of the normal cumulative distribution function by completing squares
and identifying the points $\xi^{(i)}_\star, 1\le i \le N$ (if they
exist), where we have $\frac{1}{d}\,\sum_{\ell=1}^{d} S_{t_\ell}((A
(\xi^{(i)}_\star,\bx^{(i)})^{\top})_\ell)=K$. Finding the points
$\xi^{(i)}_\star, 1\le i \le N$, is not a difficult numerical task since
each of them can be obtained as the root of an equation defined by a
univariate convex function, for which Newton's method converges in few
steps to a satisfactorily accurate solution.

\begin{figure}[t]
\begin{center}
\includegraphics[width=16cm, height=8cm]{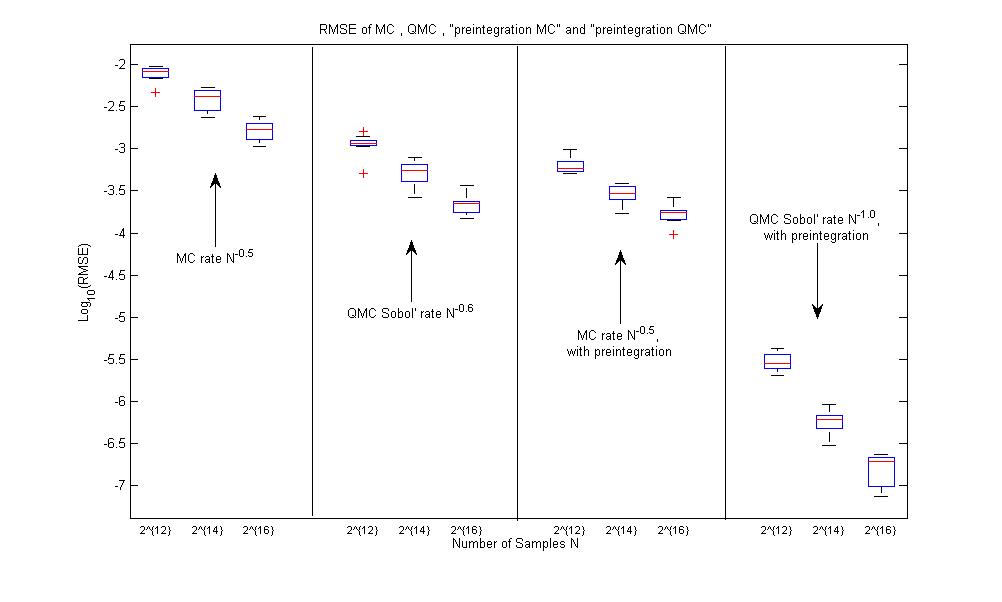}
\end{center}
\vspace{-1cm}
\caption{Root mean square errors for (from left to right) Monte Carlo,
Quasi Monte Carlo, preintegrated Monte Carlo and preintegrated Quasi Monte Carlo}
\label{fig:PCA}
\end{figure}

The parameters in our tests were fixed to $K=100,S_0=100, r=0.1 ,
\sigma=0.1, T=1$. We summarize our numerical experiments in
Figure~\ref{fig:PCA}. In the figure we show the box-plots of the
$\log_{10}$ of relative root mean square error (RMSE), each obtained from
10 independent random replications, with PCA factorization of covariance
matrix for the arithmetic average digital Asian option. Results are shown
in four groups containing three box-plots each. Each group corresponds to
one of the following method: in order, MC, QMC, MC with preintegration and
QMC with preintegration. In each group we have three box-plots to
characterize the error convergence, each box-plot containing RMSE sampled
with a given sample size. For all integration methods we chose the sample
sizes $N=2^{12},2^{14},2^{16}$. Note that for MC and QMC we generate
samples over $\R^d$, while for the preintegration MC and QMC we generate
samples only over $\R^{d-1}$.

The results show that randomized QMC exhibits higher convergence than MC,
but the convergence rate is still not optimal ($\sim N^{-0.6}$). When we
combine the preintegration method with MC, we observe an improvement in
the implied error constant, as predicted, but the convergence rate remains
the same as MC ($= N^{-0.5}$), as of course it should. Combining the
preintegration method with randomized QMC reduces the error
satisfactorily, and improves the convergence rate to close to the best
possible rate $N^{-1.0}$.

\section{Variance reduction by preintegration} \label{sec:var}

In this section we consider the space $\calL_{2,\rho_d}$ of
square-integrable functions on $\bbR^d$, with $\rho_d$-weighted $\calL_2$
inner product and norm.

The preintegration step \eqref{firstPj} can be viewed more generally as a
projection, which is the key operation underlying the well-known ANOVA
decomposition. For a general function $g\in \calL_{2,\rho_d}$ the ANOVA
decomposition takes the form \cite{Sobol90}
\begin{equation}\label{anova}
 g \,=\, \sum_{\setu\subseteq \setD} g_\setu,
\end{equation}
where the sum is over all the $2^d$ subsets of $\setD:=\{1,\ldots,d\}$,
and each term $g_\setu$ depends only on the variables $x_k$ with
$k\in\setu$, and with the additional property that the projection operator
$P_k$ defined by (as in \eqref{firstPj})
\begin{equation*}
  (P_k g)(\bsx_{-k}) \,:=\, \int_{-\infty}^\infty g(x_k, \bsx_{-k})\,\rho(x_k)\,\rd x_k
\end{equation*}
annihilates all ANOVA terms $g_\setu$ with $k\in\setu$:
\begin{equation}\label{Pkprop}
 P_k g_\setu =0 \quad\mbox{for}\quad k\in\setu, \quad \mbox{whereas}\quad
 P_k g_\setu =g_\setu \quad\mbox{for}\quad k\notin\setu.
\end{equation}
The ANOVA terms can be written explicitly as \cite{KSWW10}
\[
  g_\setu \,=\, \sum_{\setv\subseteq\setu} (-1)^{|\setu|-|\setv|} \bigg(\prod_{k\notin\setv} P_k\bigg) g.
\]

It follows from \eqref{Pkprop}, since $I_d$ involves integration with
respect to every variable $x_k$ for $k\in\setD$, that
\[
 I_d g \,=\, g_\emptyset.
\]
Another consequence is that the ANOVA terms are orthogonal in $\calL_{2,
\rho_d}$,
\[
 \int_{\R^d} g_\setu(\bsx)\,g_\setv(\bsx)\,\rho_d(\bsx)\,\rd \bsx = 0
 \quad \mbox{for} \quad \setu\ne\setv.
\]
As a result, the variance of $g$ has the well known property that it is a
sum of the variances of the separate ANOVA terms,
\begin{equation} \label{eq:var}
  \sigma^2(g)
  \,=\, \sum_{\emptyset\ne \setu\subseteq{\setD}} \sigma^2(g_\setu),
\end{equation}
where
\[
  \sigma^2(g) \,:=\,
  \int_{\R^d} g^2(\bsx)\,\rho_d(\bsx)\,\rd \bsx - g_\emptyset^2
  \qquad\mbox{and}\qquad
  \sigma^2(g_\setu) \,=\, \int_{\R^d} g_\setu^2(\bsx)\, \rho_d(\bsx) \,\rd \bsx
  \quad \mbox{for}\quad \setu\ne\emptyset.
\]

With these preparations, we are now ready to make a simple observation
that preintegration is a variance reduction strategy for any general
$\calL_2$ function, not specific to our functions with kinks or jumps.
This explains why the preintegration strategy improves the performance of
MC methods.

\begin{lemma} \label{lem:var}
The projection $P_k$ reduces the variance of $g$ for all $g \in
\calL_{2.\rho_d}$ and all $k\in\setD$.
\end{lemma}

\begin{proof}
For any $g\in\calL_{2,\rho_d}$ and any $k\in\setD$, it follows from
\eqref{anova} and \eqref{Pkprop} that
\[ 
  P_k g \,=\, \sum_{k\notin\setu\subseteq\setD} g_\setu,
\] 
that is, the operation $P_k$ applied to $g$ has the effect of annihilating
those ANOVA terms $g_\setu$ of $g$ with $k\in\setu$. As a result, the
ANOVA terms of the resulting function $P_kg$ are precisely the ANOVA terms
$g_\setu$ of $g$ for which $k\notin\setu$. Hence we have
\begin{equation} \label{eq:var-Pk}
 \sigma^2(P_k g)
 \,=\, \sum_{k\notin\setu\subseteq{\setD},\,\setu\ne\emptyset} \sigma^2(g_\setu).
\end{equation}
The result follows by comparing \eqref{eq:var-Pk} with \eqref{eq:var}.
\end{proof}

\section{Smoothing by preintegration} \label{sec:smooth}

In this section we first slightly generalize the mathematical setting from
\cite{GKS13}, providing some details on Sobolev spaces and weak
derivatives which are needed for the formulation of our main smoothing
theorems. Then we establish two new smoothing theorems for these Sobolev
spaces, extending \cite[Theorem~1]{GKSnote} from kink to jumps.

\subsection{Sobolev spaces with generalized weight functions} \label{sec:sob}

Following \cite[Section~2.2]{GKS13}, for $j\in\setD$ let $D_j$ denote the
partial derivative operator
\[
  (D_jg)(\bsx) \,=\, \frac{\partial g}{\partial x_j}(\bsx).
\]
Throughout this paper, the term \emph{multi-index} refers to a vector
$\bsalpha=(\alpha_1,\ldots,\alpha_d)$ whose components are nonnegative
integers, and we use the notation $|\bsalpha| = \alpha_1+\cdots+\alpha_d$
to denote the sum of its components. For any multi-index
$\bsalpha=(\alpha_1,\ldots,\alpha_d)$, we define
\begin{equation} \label{Dalpha}
  D^\bsalpha \,=\, \prod_{j=1}^d D_j^{\alpha_j}
  \,=\, \prod_{j=1}^d \left(\frac{\partial}{\partial x_j}\right)^{\alpha_j}
  \,=\, \frac{\partial^{|\bsalpha|}}{\prod_{j=1}^d \partial
  x_j^{\alpha_j}},
\end{equation}
and we say that the derivative $D^\bsalpha f$ is of order $|\bsalpha|$.

Let $\calC(\bbR^d) = \calC^0(\bbR^d)$ denote the linear space of
continuous functions defined on $\bbR^d$. For a nonnegative integer $r\ge
0$, we define $\calC^r(\bbR^d)$ to be the space of functions whose
classical derivatives of order $\le r$ are all continuous at every point
in $\bbR^d$, with no limitation on their behavior at infinity. For
example, the function $g(\bsx) = \exp(\sum_{j=1}^d x_j^2)$ belongs to
$\calC^r(\bbR^d)$ for all values of $r$. For convenience we write
$\calC^\infty(\bbR^d) = \cap_{r\ge 0} \calC^r(\bbR^d)$.

In addition to classical derivatives, we shall need also \emph{weak}
derivatives. By definition, the weak derivative $D^\bsalpha g$ is a
measurable function on $\bbR^d$ which satisfies
\begin{equation} \label{weak}
  \int_{\bbR^d} (D^\bsalpha g)(\bsx)\,v(\bsx)\,\rd\bsx
  \,=\, (-1)^{|\bsalpha|} \int_{\bbR^d} g(\bsx)\,(D^\bsalpha v)(\bsx)\,\rd\bsx
  \quad\mbox{for all}\quad
  v\in \calC^\infty_0(\bbR^d),
\end{equation}
where $\calC^\infty_0(\bbR^d)$ denotes the space of infinitely
differentiable functions with compact support in $\bbR^d$, and where the
derivatives on the right-hand side of \eqref{weak} are classical partial
derivatives. It can be shown, using the definition \eqref{weak}, that
$D_jD_k = D_kD_j$ for all $j,k\in\setD$, thus the ordering of the weak
first derivatives that make up $D^\bsalpha$ in \eqref{Dalpha} is
irrelevant.

If $g$ has classical continuous derivatives up to order $|\bsalpha|$, then
they satisfy \eqref{weak}, which in the classical sense is just the
integration by parts formula on $\bbR^d$.  Unless stated otherwise, the
derivatives in this paper are weak derivatives, which in principle allows
the possibility that they are defined only ``almost everywhere''. However,
a recurring theme is that our weak derivatives are shown to be continuous
(or strictly, can be represented by continuous functions), in which case
the weak derivatives are at the same time classical derivatives.

We now turn to the definition of the function spaces. For $p \in
[1,\infty]$, we first define weighted $\calL_p$ norms:
\[ 
  \|g\|_{\calL_{p,\widetilde\rho_d}} \,=\,
  \begin{cases}
  \left(\int_{\bbR^d} |g(\bsx)|^p\,\widetilde\rho_d(\bsx)\,\rd\bsx\right)^{1/p}
  & \mbox{if } p\in [1,\infty), \\
  \esup_{\bsx\in \bbR^d} |g(\bsx)| & \mbox{if } p = \infty
  \end{cases}
\] 
where $\widetilde\rho_d$ is a positive integrable function on $\R^d$.

When dealing with function spaces of derivatives of a function $g$, it
turns out to be convenient to allow flexibility in the choice of weight
function~$\widetilde\rho_d$. We therefore generalize the setting in
\cite{GKS13} and introduce a family $\zeta_{d,\bsalpha}$ of such weight
functions, one for each derivative $D^{\bsalpha}$, given by
\begin{equation}\label{zetaprod}
  \zeta_{d,\bsalpha}(\bsx)
  \,:=\, \prod_{k=1}^d \zeta_{\alpha_k}(x_k),
\end{equation}
where $\{\zeta_i\}_{i\ge 0}$ is a sequence of continuous integrable functions on
$\bbR$, satisfying
\begin{equation}\label{zbiggerrho}
 \rho(x) \,=\, \zeta_0(x) \,\le\, \zeta_1(x) \,\le\, \zeta_2(x) \,\le\, \cdots \qquad\mbox{for all}\quad x\in\R.
\end{equation}
The intuitive idea is that higher derivatives with respect to every
coordinate need to be limited in their growth towards infinity by making
$\zeta_{i}$ decay more slowly for larger order of derivatives $i$.

With these generalized weight functions $\zeta_{d,\bsalpha}$, denoted
collectively by $\bszeta$, we consider two kinds of Sobolev space: the
\emph{isotropic Sobolev space} with smoothness parameter $r\ge 0$, for $r$
a nonnegative integer,
\[ 
 \calW_{d,p,\bszeta}^r \,=\,
 \left\{ g \,:\, D^\bsalpha g \in \calL_{p,\zeta_{d,\bsalpha}}
 \quad\mbox{for all}\quad |\bsalpha| \le r \right\},
\] 
and the \emph{Sobolev space of dominating mixed smoothness} with
smoothness multi-index $\bsr=(r_1,\ldots,r_d)$,
\[ 
  \calW_{d,p,\bszeta,\mix}^\bsr \,=\,
  \left\{ g \,:\, D^\bsalpha g \in \calL_{p,\zeta_{d,\bsalpha}}
  \quad\mbox{for all}\quad \bsalpha\le\bsr \right\},
\] 
where $\bsalpha\le\bsr$ is to be understood componentwise, and the
derivatives are weak derivatives. For convenience we also write
$\calW_{d,p,\bszeta}^0 = \calL_{p,\rho_d}$ and $\calW_{d,p,\bszeta}^\infty
= \cap_{r\ge 0} \calW_{d,p,\bszeta}^r$. Analogously, we define
$\calC^{\bsr}_{\mix}(\bbR^d)$ to be the space of functions $g$ whose
classical derivatives $D^\bsalpha g$ with $\bsalpha\le \bsr$ are all
continuous at every point in $\bbR^d$, with no limitation on their
behavior at infinity.

The norms corresponding to the two kinds of Sobolev space can be defined,
for example, as in the classical sense, by
\[ 
 \|g\|_{\calW_{d,p,\bszeta}^r} =
  \Bigg(\sum_{|\bsalpha|\le r}
  \|D^\bsalpha g\|_{\calL_{p,\zeta_{d,\bsalpha}}}^2\Bigg)^{1/2}
  \quad\mbox{and}\quad
  \|g\|_{\calW_{d,p,\bszeta,\mix}^\bsr} =
  \Bigg(\sum_{\bsalpha\le \bsr}
  \|D^\bsalpha g\|_{\calL_{p,\zeta_{d,\bsalpha}}}^2\Bigg)^{1/2}.
\] 

We have the following relationships between the spaces:
\begin{itemize}
\item [(i)] $\calW_{d,p',\bszeta}^r \subseteq \calW_{d,p,\bszeta}^r$
    and $\calW_{d,p',\bszeta,\mix}^\bsr \subseteq
    \calW_{d,p,\bszeta,\mix}^\bsr$ for $1\le p\le p'\le \infty$.
\item [(ii)] $\calW_{d,p,\bszeta,\mix}^\bsr \subseteq
    \calW_{d,p,\bszeta}^r \iff \min_{j \in \setD} r_j \ge r \quad$ and
    $\quad \calW_{d,p,\bszeta}^r = \cap_{|\bsr|=r}
    \calW_{d,p,\bszeta,\mix}^\bsr$.
\item [(iii)] $\calW_{d,p,\bszeta,\mix}^{s,\ldots,s} \subseteq
    \calW_{d,p,\bszeta}^r \iff s\ge r \quad $ and  $\quad \calW_{d,p,\bszeta}^r
    \subseteq \calW_{d,p,\bszeta,\mix}^{s,\ldots,s} \iff r\ge s\,d$.
\item [(iv)] $\calW_{d,p,\bszeta}^r \subseteq \calC^k(\bbR^d)$ if
    $r>k+d/p$ (Sobolev embedding theorem).
\item [(v)] For $p\in [1,\infty)$ and $r\ge 1$, if $g \in
    \calW_{d,p,\bszeta}^r$ then $D^\bsalpha g\in
    \calW_{d,p,\bszeta}^{r-|\bsalpha|}$ for all $|\bsalpha| \le r$.
\item [(vi)] For $p\in [1,\infty)$ and $\bsr\ge \bsone$, if $g \in
    \calW_{d,p,\bszeta}^{\bsr}$ then $D^\bsalpha g\in
    \calW_{d,p,\bszeta}^{\bsr-\bsalpha}$ for all $\bsalpha \le \bsr$.
\end{itemize}
Properties (i)--(iv) are straightforward. Properties (v) and (vi) are a
bit more involved due to the varying generalized weight functions
considered here. Indeed, when $\overline\bsalpha$ is a multi-index
satisfying $0\le |\overline\bsalpha|\le r-|\bsalpha|$, we have
\begin{align*}
  \|D^{\overline\bsalpha} (D^\bsalpha g)\|_{\calL_{p,\zeta_{d,\overline\bsalpha}}}
  &\,=\,   \left(\int_{\bbR^d} |(D^{\overline\bsalpha} (D^\bsalpha g))(\bsx)|^p\,
  \zeta_{d,\overline\bsalpha}(\bsx)\,\rd\bsx\right)^{1/p} \\
  &\,\le\,   \left(\int_{\bbR^d} |(D^{\widehat\bsalpha} g)(\bsx)|^p\,
  \zeta_{d,\widehat\bsalpha}(\bsx)\,\rd\bsx\right)^{1/p}
  \,=\,  \|D^{\widehat\bsalpha}g\|_{\calL_{p,\zeta_{d,\widehat\bsalpha}}}
  \,<\,\infty,
\end{align*}
where we introduced $\widehat\bsalpha := \overline\bsalpha + \bsalpha$ and
we used
$\zeta_{d,\overline\bsalpha}(\bsx)\le\zeta_{d,\widehat\bsalpha}(\bsx)$
since $\overline\bsalpha\le\widehat\bsalpha$. The finiteness in the final
step follows from $g\in \calW^r_{d,p,\bszeta}$ and $|\widehat\bsalpha|\le
r$. This justifies (v). The argument can easily be modified to justify
(vi).

\subsection{New smoothing theorems} \label{sec:newthms}

In this subsection we establish two smoothing theorems: one for the
isotropic Sobolev space, the other for the mixed Sobolev space.  The
proofs are modeled on the proof of \cite[Theorem~1]{GKSnote}, but are
extended here to cover discontinuous integrands.

\begin{theorem}[Result for the isotropic Sobolev space with weight functions $\zeta_{d,\bsalpha}$] \label{thm:main1}
Let $d\ge 2$, $r\ge 1$, $p\in [1,\infty)$, and let
$\rho\in\calC^{r-1}(\bbR)$ be a strictly positive probability density
function. Let
\[ 
  f(\bsx) \,:=\, \theta(\bsx)\,{\rm ind}(\phi(\bsx)),
  \quad\mbox{where}\quad
  \theta,\phi\in\calW^r_{d,p,\bszeta}\cap\calC^r(\bbR^d),
\] 
with generalized weight functions $\zeta_{d,\bsalpha}$ satisfying
\eqref{zetaprod} and \eqref{zbiggerrho}, and with ${\rm ind}(\cdot)$
denoting the indicator function. Let $j\in\setD := \{1,\ldots,d\}$ be
fixed, and suppose that
\begin{equation} \label{phi}
  (D_j\phi)(\bsx) > 0
  \quad\forall\bsx\in \bbR^d,
  \qquad\mbox{and}\qquad
  \phi(\bsx) \to \infty \quad\mbox{as}\quad x_j\to \infty \quad\forall \bsx_{-j} \in \bbR^{d-1}.
\end{equation}
Writing $\bsy := \bsx_{-j}$ so that $\bsx = (x_j,\bsy)$, let
\[ 
  U_j \,:=\,
  \{
  \bsy\in\bbR^{d-1}
  \,:\,
  \phi(x_j,\bsy)=0
  \mbox{ for some }x_j\in\bbR
  \}
 \qquad\mbox{and}\qquad
 U_j^+ \,:=\, \bbR^{d-1}\setminus U_j.
\] 
If $U_j$ is empty, then $f = \theta$. If $U_j$ is not empty, then $U_j$ is
open, and there exists a unique function $\psi \equiv \psi_j\in
\calC^r(U_j)$ such that $\phi(x_j,\bsy)=0$ if and only if $x_j =
\psi(\bsy)$ for $\bsy\in U_j$. In the latter case we assume that every
function of the form
\begin{align} \label{h}
  \begin{cases}
  h(\bsy) \,=\,
  \displaystyle\frac{(D^{\bseta}\theta)(\psi(\bsy),\bsy)\,\prod_{i=1}^a
  [(D^{\bsgamma^{(i)}}\phi)(\psi(\bsy),\bsy)]}{[(D_j\phi)(\psi(\bsy),\bsy)]^b}\,
  \rho^{(c)}(\psi(\bsy)),
  \quad \bsy\in U_j,
  \vspace{0.2cm} \\
  \mbox{where $a,b,c$ are integers and $\bsgamma^{(i)}$, $\bseta$ are
  multi-indices with the constraints}
  \vspace{0.1cm} \\
 1\le a,b\le 2r-1,\quad
 1\le |\bsgamma^{(i)}|\le r,\quad
 0\le |\bseta|,c\le r-1, \quad
 1\le |\bsgamma^{(i)}| + |\bseta| + c \le r,
 \end{cases}
\end{align}
satisfies both
\begin{equation} \label{h1}
  h(\bsy)\to 0 \qquad
  \mbox{as}\quad \mbox{$\bsy$ approaches a boundary point of $U_j$ lying in $U_j^+$},
\end{equation}
and
\begin{equation}\label{h2}
  \int_{U_j} |h(\bsy)|^p\, \zeta_{d-1,\bsalpha_{-j}}(\bsy)\,\rd\bsy  \,<\, \infty,
  \quad  \text{ for all } |\bsalpha|\le r  \text{ with } \alpha_j=0,
\end{equation}
where $\bsalpha_{-j}$ denotes the multi-index with $d-1$ components
obtained from $\bsalpha$ by leaving out $\alpha_j$. Then
\[
  P_j f \in \calW^r_{d-1,p,\bszeta} \cap\calC^r(\bbR^{d-1}).
\]
\end{theorem}

\begin{proof}
We defer the proof of this theorem to Section~\ref{sec:proof1}.
\end{proof}

In effect, under the conditions in the theorem, the single integration
with respect to $x_j$ is sufficient to ensure that $P_jf$ inherits the
full smoothness of $\theta$ and $\phi$.

We remark that when $\theta = \phi$ we are back at the same function
$f(\bsx) = \max(\phi(\bsx,0))$ as considered in \cite[Theorem~1]{GKSnote}.
However, for this case we see that the new result is not as sharp as the
old one in the sense that the upper bounds on the values of $a$, $b$, $c$,
$|\bsgamma^{(i)}|$ in the condition \eqref{h} are larger than those in
\cite[Theorem~1]{GKSnote}. This is because the explicit prior knowledge of
$\theta = \phi$ means that we know a certain term vanishes (precisely, the
second term on the right-hand side of \cite[Formula~(11)]{GKSnote}). This
observation also indirectly explains how the new result for jumps require
stronger conditions on the functions $\theta$, $\phi$ and $\rho$ than the
corresponding result for kinks.

The conditions \eqref{h1} and \eqref{h2} in the theorem are difficult to
verify directly because the function $h$ depends explicitly on the inverse
function $\psi(\bsy)$.  Fortunately, a sufficient condition for \eqref{h1}
to hold is that
\begin{align} \label{suff_new}
  &\left|\frac{(D^{\bseta}\theta)(x_j,\bsy)\,\prod_{i=1}^a
  [(D^{\bsgamma^{(i)}}\phi)(x_j,\bsy)]}{[(D_j\phi)(x_j,\bsy)]^b}\,
  \rho^{(c)}(x_j)\right|
  \,\le\, E_1(x_j) E_2(\bsy),
\end{align}
where $E_1,E_2$ are positive functions satisfying
\begin{itemize}
 \item $E_1$ is bounded and $E_1(x_j)\to 0$ as $x_j\to - \infty$,
 \item $E_2$ is locally bounded (bounded over compact sets) and
     $\int_{\bbR^{d-1}} \left|E_2(\bsy) \right|^p
  \zeta_{d-1,\bsalpha_{-j}}(\bsy)\,\rd\bsy \,<\, \infty$ for all
  $|\bsalpha|\le r$.
\end{itemize}
Considering a point $\bsy^\star$ on the boundary $\varGamma(U_j) \subset
U_j^+$, and a sequence $(\bsy_n)_{n \in \mathbb{N}} \subset U_j $ such
that $\bsy_n \to \bsy^\star$, we see that $E_2(\bsy_n)$ is bounded and
$E_1(\psi(\bsy_n)) \to 0$ since $\psi(\bsy_n) \to -\infty$. Therefore
\eqref{suff_new} is sufficient for \eqref{h1}. Moreover, for
$|\bsalpha|\le r$ we have
\begin{align*}
&\int_{U_j}
\left | \displaystyle\frac{(D^{\bseta}\theta)(\psi(\bsy),\bsy)\,\prod_{i=1}^a
  [(D^{\bsgamma^{(i)}}\phi)(\psi(\bsy),\bsy)]}{[(D_j\phi)(\psi(\bsy),\bsy)]^b}\,
  \rho^{(c)}(\psi(\bsy)) \right |^p \zeta_{d-1,\bsalpha_{-j}}(\bsy)\,\rd\bsy \\
  &\,\le\, \int_{U_j} \left|E_1(\psi(\bsy)) \right|^p \left|E_2(\bsy) \right|^p
  \zeta_{d-1,\bsalpha_{-j}}(\bsy)\,\rd\bsy
  \,\le\, B \int_{U_j} \left|E_2(\bsy) \right|^p
  \zeta_{d-1,\bsalpha_{-j}}(\bsy)\,\rd\bsy
  \,<\, \infty,
  \end{align*}
for some positive constant $B$. Therefore \eqref{suff_new} is also
sufficient for \eqref{h2}.

We can also deduce a result for Sobolev spaces of dominating mixed
smoothness.

\begin{theorem}[Result for the Sobolev space of dominating mixed smoothness] \label{thm:main2}
Let $d\ge 2$, $p\in [1,\infty)$, $j\in \setD$, and let
$\rho\in\calC^{r-1}(\bbR)$ be a strictly positive probability density
function. Let $\bsr = (r_1,\ldots,r_d)$ be a multi-index satisfying
\[
  r_j \,\ge\, \textstyle\sum_{1\le k\le d,\,k\ne j} r_k \,\ge\, 1.
\]
If we replace the conditions on $\theta$, $\phi$ in
Theorem~\ref{thm:main1} by
\[
  \theta,\phi\in \calW^{\bsr}_{d,p,\bszeta,\mix}\cap\calC^\bsr_{\mix}(\bbR^d),
\]
and further restrict \eqref{h}--\eqref{h2} to functions $h$ with
multi-indices $\bsgamma^{(i)} \le\bsr$, $\bseta < \bsr$, and
$\bsalpha\le\bsr$, then the conclusion becomes:
\[
  P_j f \in
  \calW^{\bsr_{-j}}_{d-1,p,\bszeta}\cap\calC^{\bsr_{-j}}_{\bmix}(\bbR^{d-1}).
\]
\end{theorem}

\begin{proof}
The proof is obtained from minor modifications of the proof of
Theorem~\ref{thm:main1} in Section~\ref{sec:proof1}. In particular, the
requirement that $r_j$ is greater than or equal to the sum of the
remaining $r_k$ for $k\ne j$ is needed because, for any multi-index
$\bsalpha \le \bsr$ with $\alpha_j = 0$, it is clear from
\eqref{chainrule} and a generalization of \eqref{dldkpjf} that the
expression for $D^\bsalpha P_jf$ includes some terms that depend on
$D_j^{|\bsalpha|}\phi$ and some terms that depend on
$D_j^{|\bsalpha|-1}\theta$.
\end{proof}

\section{Applying the theory to option pricing} \label{sec:apply}

We now apply our results to the option pricing example.

Recall from Section~\ref{sec:app} that after PCA factorization the
function $f$ from the digital option pricing example takes the form
\eqref{problem2}, with $\theta$ a constant function and $\phi$ given by
\eqref{phi_explicit}. It follows that
\[
  (D_j \phi)(\bsx) \,=\,
  \frac{\sigma\,S_0}{d} \sum_{\ell=1}^d
  \exp\Bigg( \left(\mu-\tfrac{\sigma^2}{2}\right) \ell\Delta t
  +\sigma \sum_{k=1}^d A_{\ell k}\, x_k\Bigg) A_{\ell j}.
\]
In particular, we see that $(D_1 \phi)(\bsx)>0$ because, as explained in
Section~\ref{sec:app}, $A_{\ell 1}>0$ for all $\ell$, thus in this case it
is appropriate to take $j=1$ in Theorem~\ref{thm:main1}.  It is also clear
that $\phi$ is in $\calC^r(\R^d)$ for all $r\in\mathbb{Z}^+$.
Additionally, we may take all the weight functions $\zeta_i$ in
\eqref{zbiggerrho} equal to the standard normal density $\rho$.  It is
then clear that the sufficient condition \eqref{suff_new} is satisfied,
and moreover that all the integrability and decay conditions in Theorem
\ref{thm:main1} are satisfied, because all derivatives of $\phi$ are
``killed'' at infinity by the Gaussian weight functions and their
derivatives.  It then follows from Theorem \ref{thm:main1} that
\[
P_1f\in \calW_{d-1,p,\boldsymbol{\rho}}^r\cap \calC^r(\R^{d-1})\quad
\forall \; r\in\mathbb{Z}^+, \quad \forall \; p\in[1,\infty).
\]

\section{Proof of the main smoothing theorem} \label{sec:proof1}

Before we proceed to prove Theorem~\ref{thm:main1}, we quote three
theorems from \cite[Section~2.4]{GKS13}, but state them with respect to
the Sobolev spaces defined with generalized weight functions
$\zeta_{d,\bsalpha}$. We outline the subtle additional steps needed in the
proofs of \cite{GKS13} to allow for this generalization.

The classical Leibniz theorem allows us to swap the order of
differentiation and integration. In this paper we need a more general form
of the Leibniz theorem as given below.

\begin{theorem}[The Leibniz Theorem {\cite[Theorem~2.1]{GKS13}}] \label{thm:swap}
Let $p\in [1,\infty)$. For $g \in \calW_{d,p,\bszeta}^1$ with generalized
weight functions $\zeta_{d,\bsalpha}$ satisfying \eqref{zetaprod} and
\eqref{zbiggerrho}, we have
\[
  D_k P_j g \,=\, P_j D_k g
  \qquad\mbox{for all}\quad j,k \in \setD \quad\mbox{with}\quad j \ne k.
\]
\end{theorem}

\begin{proof}
We follow the proof of \cite[Theorem~2.1]{GKS13} to the last paragraph
where Fubini's theorem was applied a second time. This application of Fubini's Theorem is valid
because
\begin{align*}
  &\left|\int_{\bbR^d} \int_{-\infty}^\infty (D_kg)(t_j,\bsx_{-j})\,\rho(t_j)\,\rd t_j\, v(\bsx)\,\rd\bsx\right| \\
  &\,\le\, \int_{\bbR^{d-1}} \int_{-\infty}^\infty
   |(D_kg)(t_j,\bsx_{-j})|\,\zeta_1(t_j)\,\rd t_j\,\rho_{d-1}(\bsx_{-j})
  \frac{\int_{-\infty}^\infty|v(x_j,\bsx_{-j})|\,\rd x_j}{\rho_{d-1}(\bsx_{-j})}\, \,\rd\bsx_{-j} \\
  &\,\le\, \| D_kg \|_{\calL_{1,\zeta_{d,\bse_k}}}
  \frac{\sup_{\bsx_{-j}\in V}\int_{-\infty}^\infty|v(x_j,\bsx_{-j})|\,\rd x_j}
       {\inf_{\bsx_{-j}\in V}\rho_{d-1}(\bsx_{-j})}
  \,<\, \infty,
\end{align*}
where $\bse_k$ is the multi-index consisting of $1$ in the position $k$,
and $0$ elsewhere, and where we made use of $\rho(t_j)\le \zeta_1(t_j)$
and $g\in \calW^1_{d,1,\bszeta}$, and that the set $V$ defined in the
proof of \cite[Theorem~2.1]{GKS13} is a compact set because of the
compactness of ${\rm supp}(v)$. The remainder of that proof then stands.
\end{proof}

The next theorem is an application of the Leibniz theorem; it establishes
that $P_jf$ inherits the smoothness of $g$.

\begin{theorem}[The Inheritance Theorem {\cite[Theorem~2.2]{GKS13}}] \label{thm:inher}
Let $r\ge 0$ and $p\in [1,\infty)$. For $g \in \calW_{d,p,\bszeta}^r$ with
generalized weight functions $\zeta_{d,\bsalpha}$ satisfying
\eqref{zetaprod} and \eqref{zbiggerrho}, we have
\[
  P_j g \,\in\, \calW_{{d-1},p,\bszeta}^r
  \qquad\mbox{for all}\quad j \in \setD.
\]
\end{theorem}

\begin{proof}
For the case $r=0$ the proof is exactly the same as the proof of
\cite[Theorem~2.2]{GKS13}. Consider now $r\ge 1$. Let $j\in\setD$ and let
$\bsalpha$ be any multi-index with $|\bsalpha|\le r$ and $\alpha_j = 0$.
Since now $g\in \calW_{d,p,\bszeta}^r$ with generalized weight functions,
we have $\|D^\bsalpha g\|_{\calL_{p,\zeta_{d,\bsalpha}}}<\infty$. To show
that $P_jg \in \calW_{d-1,p,\bszeta}^r$ we need to show that $\|D^\bsalpha
P_j g\|_{\calL_{p,\zeta_{d-1,\bsalpha_{-j}}}}<\infty$. Mimiking the proof
of \cite[Theorem~2.2]{GKS13}, we write successively
\begin{align} \label{eq:step}
  D^\bsalpha P_j g
  &\,=\, \left( \textstyle\prod_{i=1}^{|\bsalpha|} D_{k_i}\right) P_j g
  \,=\, \left( \textstyle\prod_{i=2}^{|\bsalpha|} D_{k_i}\right) P_j D_{k_1} g \nonumber \\
  &\,=\, \cdots
  \,=\, D_{k_{|\bsalpha}|} P_j
       \left( \textstyle\prod_{i=1}^{|\bsalpha|-1} D_{k_i} \right) g
  \,=\, P_j \left( \textstyle\prod_{i=1}^{|\bsalpha|} D_{k_i} \right) g
  \,=\, P_j D^\bsalpha g,
\end{align}
where $k_i\in\setD\setminus\{j\}$ and $k_1,\ldots,k_{|\bsalpha|}$ need not
be distinct. Each step in \eqref{eq:step} involves a single
differentiation under the integral sign, and is justified by the Leibniz
theorem (Theorem~\ref{thm:swap}) because we know from the property (v) in
Subsection~\ref{sec:sob} that $(\prod_{i=1}^{\ell} D_{k_i})
g\in\calW_{d,p,\bszeta}^{\,r-\ell} \subseteq \calW_{d,p,\bszeta}^1$ for
all $\ell\le |\bsalpha|-1\le r-1$. We have therefore
\begin{align*}
  &\|D^\bsalpha P_j g\|_{\calL_{p,\zeta_{d-1,\bsalpha_{-j}}}}
  \,=\, \|P_j D^\bsalpha g\|_{\calL_{p,\zeta_{d-1,\bsalpha_{-j}}}} \\
  &\,=\,   \Bigg(\int_{\bbR^{\setD\setminus\{j\}}}
  \left|
  \int_{-\infty}^\infty (D^\bsalpha g)(\bsx)\,\rho(x_j)\,\rd x_j
  \right|^p
  \,\zeta_{d-1,\bsalpha_{-j}}(\bsx_{-j})\,\rd\bsx_{-j}
  \Bigg)^{1/p} \\
  &\,\le\,  \Bigg(\int_{\bbR^{\setD\setminus\{j\}}}
  \left(
  \int_{-\infty}^\infty |(D^\bsalpha g)(\bsx)|^p\,\rho(x_j)\,\rd x_j
  \right)
  \,\zeta_{d-1,\bsalpha_{-j}}(\bsx_{-j})\,\rd\bsx_{-j}
  \Bigg)^{1/p}
  \,=\, \|D^\bsalpha g\|_{\calL_{p,\zeta_{d,\bsalpha}}} \,<\, \infty,
\end{align*}
where we applied H\"older's inequality to the inner integral as in \cite[Equation~(2.11)]{GKS13}
and used $ \zeta_{\alpha_j}(x_j) = \zeta_0(x_j)=\rho(x_j)$. This
completes the proof.
\end{proof}

The implicit function theorem stated below is crucial for the main results
of this paper. In the rest of the paper, for any $r\ge 0$, $k\ge 1$, and
an open set $U\subset\bbR^k$, we define $\calC^r(U)$ to be the space of
functions whose classical derivatives of order $\le r$ are all continuous
at every point in $U$.

\begin{theorem}[The Implicit Function Theorem {\cite[Theorem~2.3]{GKS13}}] \label{thm:implicit}
Let $j\in\setD$. Suppose $\phi \in \calC^1(\bbR^d)$ satisfies
\begin{equation} \label{djnot}
  (D_j \phi) (\bsx) \,>\, 0 \qquad\mbox{for all}\quad \bsx \in \bbR^d.
\end{equation}
Let
\[
  U_j \,:=\, \{ \bsx_{-j}\in\bbR^{d-1} \,:\,
  \phi(x_j,\bsx_{-j})=0 \mbox{ for some $($unique$)$ } x_j\in\bbR\}.
\]
If $U_j$ is not empty then there exists a unique function $\psi_j \in
\calC^1 (U_j)$ such that
\[
  \phi (\,\psi_j (\bsx_{-j}),\bsx_{-j}) \,=\, 0
  \qquad\mbox{for all}\quad \bsx_{-j}
  \in U_j,
\]
and for all $k \ne j$ we have
\begin{equation} \label{dkpsi-first}
  (D_k\psi_j) (\bsx_{\setD\setminus\{j\}})
  \,=\, - \frac{(D_k\phi)(\bsx)}{(D_j \phi)(\bsx)}\;
  \bigg|_{\;x_j \,=\, \psi_j(\bsx_{\setD\setminus\{j\}})}
  \quad\mbox{for all}\quad \bsx_{\setD\setminus\{j\}}
  \in U_j.
\end{equation}
If in addition $\phi \in \calC^r(\bbR^d)$ for some $r\ge 2$, then $\psi_j
\in \calC^r (U_j)$.
\end{theorem}

Note that the derivatives in the implicit function theorem are classical
derivatives, and the condition \eqref{djnot} needs to hold for \emph{all}
$\bsx\in\bbR^d$.

We are almost ready to prove Theorem~\ref{thm:main1}. But first we give a
remark and a couple of auxiliary results.

\begin{remark}
It is easily seen that $U_j$ and $U_j^+$ in Theorem~\ref{thm:main1} can
also be defined by
\begin{align*}
  U_j
  &\,=\, \left\{\bsy\in\bbR^{d-1}:\lim_{x_j\to -\infty}\phi(x_j,\bsy)<0\right\},
  \\
  U_j^+
  &\,=\, \left\{\bsy\in\bbR^{d-1}:\lim_{x_j\to -\infty}\phi(x_j,\bsy)\ge 0\right\}.
\end{align*}
\end{remark}

In the proof of the theorem we make essential use of the following lemma.
This result is needed to ensure that all the derivatives we encounter are
continuous across the boundary between $U_j$ and~$U_j^+$.

\begin{lemma}\label{lem:toinf}
Under the condition \eqref{phi}, the function $\psi_j:\R\to\R$ has the
following property
\begin{equation}\label{psitomininfty}
\psi_j(\bsy)\to -\infty
\end{equation}
as $\bsy$ approaches a point on the boundary of $U_j$.
\end{lemma}

\begin{proof}
Consider a point $\bsy^\star$ a point on the boundary of $U_j$, and hence
(because $U_j$ is open) lying in $U_j^+$. Consider also a sequence
$(\bsy_n)_{n \in \mathbb{N}}\subset U_j$ with $\bsy_n \rightarrow
\bsy^\star$ as $n \rightarrow \infty$. We assert that the sequence
$(\psi(\bsy_n))_{n \in \mathbb{N}} $ has no accumulation points in $\R$.
This is true because if we assume otherwise then there would exist a
convergent subsequence $(\psi(\bsy_{n_k}))_{k \in \mathbb{N}} $, with
$\psi(\bsy_{n_k}) \rightarrow x_j^\star$ as $k \rightarrow \infty$ for
$x_j^\star \in \R$. But because of the continuity of $\phi$ we must have
\[
\phi(x_j^\star,\bsy^\star)=\lim_{k\rightarrow \infty}\phi(\psi(\bsy_{n_k}), \bsy_{n_k}) =  0,
\]
since by definition $\phi(\psi(\bsy_{n_k}), \bsy_{n_k}) = 0, \forall k \in \mathbb{N}$.  This implies that $\bsy^\star
\in U_j$, which is a contradiction. Therefore the sequence
$(\psi(\bsy_n))_{n \in \mathbb{N}} $ has no accumulation points in $\R$.
This implies (due to the Bolzano-Weierstrass Theorem) that
$(\psi(\bsy_n))_{n \in \mathbb{N}}\cap [a,b]$ is a finite set, for each
interval $[a,b], a,b \in \R$. Thus,
\[
 \lim_{n \rightarrow \infty} \left| \psi(\bsy_n) \right| = \infty.
\]

To eliminate the possibility that $+\infty$ is an accumulation point of
$\psi(\bsy_n)$ we observe that, due to condition \eqref{unbounded}, and
with $\bsy^\star$ as above, there exists an $x_j^\star$ such that
$\phi(x_j^\star,\bsy^\star)>0$. Because of the continuity of $\phi$, there
is a ball around the point $(x_j^\star,\bsy^\star)$, denoted by
$B(x_j^\star,\bsy^\star)$ such that $\phi$ is positive for each point in
$B(x_j^\star,\bsy^\star)$. Assume now that we have a subsequence $
(\psi(\bsy_{n_m}))_{m \in \mathbb{N}}$ such that $\lim_{m \rightarrow
\infty} \psi(\bsy_{n_m}) = +\infty$. Because $\bsy_{n_m}$ converges to
$\bsy^\star$ as $m \rightarrow \infty$, it follows  that
$(x_j^\star,\bsy_{n_m}) \in B(x_j^\star,\bsy^\star) $ for all $m$
sufficiently large. But assumption $\psi(\bsy_{n_m}) \to +\infty$, and
the monotonicity condition in \eqref{phi} implies that
for all $m$ sufficiently large we have $\psi(\bsy_{n_m})>x_j^\star$, and therefore $0 <
\phi(x_j^\star,\bsy_{n_m})<\phi(\psi(\bsy_{n_m}),\bsy_{n_m})=0$, which is
clearly a contradiction. Therefore we conclude that
\[
 \lim_{n \rightarrow \infty}  \psi(\bsy_n)  = -\infty.
\]
\end{proof}

Another auxiliary result is needed to show that the assumption \eqref{h1}
implies continuous differentiability of $P_j f$ at boundary points of
$U_j$.

\begin{lemma}\label{lem:Crproperty}
Let $r\ge 0$ and $k\ge 1$. Suppose $g \in \calC^r(U)$ for some open domain
$U \subset \R^k$ and $g(\bsy)=0$ for all $\bsy \in U^c$. Suppose that for
any multi-index $\bsalpha$ with $|\bsalpha| \leq r$ and any sequence
$(\bsy_n)_{n \in \mathbb{N}}\subset U$ we have
\begin{equation}\label{Crprop:Assumption}
   \lim_{n \rightarrow \infty} \bsy_n = \bsy^\star \mbox{ with }  \bsy^\star \in U^c     \quad \Rightarrow  \quad
\lim_{n \rightarrow \infty} (D^{\bsalpha} g)(\bsy_n) = 0 \; .
\end{equation}
Then we have  $g \in \calC^r(\R^k)$, with $(D^\bsalpha g)(\bsy) = 0$ for
all $\bsy \in U^c$.
\end{lemma}

\begin{proof}
The statement is obviously true for $r=0$ where mere continuity of $g$ is
asserted. Now suppose it holds for a natural number $r_0 \geq 0$ and
consider any multi-index $\bsalpha = \bsalpha_0 + \bse_i$, with $\bse_i$
denoting a canonical basis vector and $|\bsalpha_0|= r_0$, and hence
$|\bsalpha| = r_0 +1 \le r$. Then we have to show that $(D^\bsalpha
g)(\bsy)$ exists at all $\bsy \in \R^k$ and is continuous at every point
of $\R^k$. For points in $U$ the derivative $D^\bsalpha g$ exists and is
continuous by assumption, and for points in the interior of $U^c$ the
derivative $D^\bsalpha g$ exists and is continuous because $g$ is zero
there. So it remains to consider the existence and continuity of
$D^\bsalpha g$ at any boundary point $\bsy^\star$, i.e., at any limit
point $\bsy^\star$ of a sequence $(\bsy_n)_{n \in \mathbb{N}} \subset U$.

To show the existence, consider the scalar valued function $D^{\bsalpha_0}
g$ which is by the induction hypothesis continuous on all of $\R^k$ and
vanishes at any boundary point $\bsy^\star$. Consider first the case $h>0$
and a point $\bsy^\star + h \bse_i$. If $\bsy^\star + h \bse_i \in U$,
then because $U$ is open and $\bsy^\star \in U^c$, we have for
\[
 \bar{h} \,:=\,
 \sup\{h'\;:\;  0\le h' < h ,\; \bsy^* + h' \bse_i \in U^c\}
\]
that $\bsy^\star + h' \bse_i \in U$ for all $\bar{h}<h'\le h$.
Furthermore, because $U^c$ is closed it follows that
$(D^{\bsalpha_0}g)(\bsy^\star+\bar{h}\bse_i) = 0 =
(D^{\bsalpha_0}g)(\bsy^\star)$, and thus we conclude from the mean value
theorem that
\[
  (D^{\bsalpha_0}g)(\bsy^\star+h\bse_i) - (D^{\bsalpha_0}g)(\bsy^\star)
  = (D^{\bsalpha_0}g)(\bsy^\star+h\bse_i) - (D^{\bsalpha_0}g)(y^\star+\bar{h}\bse_i)
  = (h-\bar{h})\, (D_i D^{\bsalpha_0}g)(\bsy^*+h^\star\bse_i)
\]
for some $h^\star$ satisfying $\bar{h}< h^\star <h$. Hence we have for the
quotient
\[
 \frac{(D^{\bsalpha_0} g)(\bsy^\star + h \bse_i) - (D^{\bsalpha_0} g)(\bsy^\star)  }{h}
 \,=\,
 \begin{cases}
 \frac{h-\bar{h}}{h}\, (D_iD^{\bsalpha_0} g)(\bsy^\star + h^\star \bse_i)
 , \, \; \bar{h}< h^\star <h,
 & \text{ if } \bsy^* + h \bse_i \in U,\\
 0 & \text{ if } \bsy^\star + h \bse_i \in U^c.
\end{cases}
\]
Then using the assumption \eqref{Crprop:Assumption}, letting $h$ be
arbitrarily small, using that $|\bar{h}|\le |h|$, and considering the
analogous situation for $h<0$, we obtain the existence of
$D_iD^{\bsalpha_0} g$ at $ \bsy^\star$, with $(D_iD^{\bsalpha_0}
g)(\bsy^\star)=0$.

To show the derivative continuity at a boundary point $\bsy^\star$,
consider any sequence $(\bsy_n)_{n \in \mathbb{N}}\subset \R^k$ with
$\lim_{n \rightarrow \infty} \bsy_n = \bsy^\star$. For a given $n$ either
$\bsy_n\in U$, in which case \eqref{Crprop:Assumption} applies, or
$\bsy_n\in U^c$, in which case $(D_i D^{\bsalpha_0}g)(\bsy_n)=0$, as
above, so that both subsequences converge to $0=(D_iD^{\bsalpha_0}
g)(\bsy^\star)$. Finally, since all partial derivatives of order $r_0+1$
are now proved continuous in $\bbR^k$, those mixed partial derivatives are
symmetric, and we can write $D_i D^{\bsalpha_0}g = D^\bsalpha g$.

Hence $D^\bsalpha g$ exists and is continuous on all of $\R^k$, i.e., the
induction step is proved. It follows then that $g\in \calC^r(\R^k)$.
\end{proof}

We are now ready to prove Theorem~\ref{thm:main1}.

\begin{proofof}{Theorem~\ref{thm:main1}}
We focus on the non-trivial case when $U_j$ is not empty. Given that
$\phi\in\calC^r(\bbR^d)$, that $(D_j\phi)(\bsx)\ne 0$ for all
$\bsx\in\bbR^d$, and that $U_j$ is not empty, it follows from the implicit
function theorem (see Theorem~\ref{thm:implicit}) 
that the set $U_j$ is open, and that there exists a unique function
$\psi\equiv\psi_j\in\calC^r(U_j)$ for which
\begin{equation} \label{psi}
  \phi(x_j,\bsy) \,=\,0
  \;\iff\;
  \psi(\bsy) \,=\, x_j
  \quad\mbox{for all }\;
  \bsy\in U_j.
\end{equation}
This justifies the existence of the function $\psi$ as stated in the
theorem.

For the function $f(\bsx) = \theta(x_j,\bsy)\,{\rm ind}(\phi(x_j,\bsy))$
we can write $P_jf$ defined by \eqref{firstPj} as
\begin{equation} \label{lim}
  (P_jf)(\bsy)
  \,=\, \int_{x_j\in \bbR \,:\, \phi(x_j,\bsy)\ge 0}
  \theta(x_j,\bsy)\,\rho(x_j)\,\rd x_j.
\end{equation}
It follows from the condition $(D_j\phi)(\bsx)> 0$ for all $\bsx\in\bbR^d$
and the continuity of $D_j\phi$ that, for fixed~$\bsy$, $\phi(x_j,\bsy)$
is a strictly increasing function of $x_j$.

We now determine the limits of integration in \eqref{lim}. If $\bsy\in
U_j^+$, then $\phi(x_j,\bsy) \ne 0$ for all $x_j\in\bbR$. Since $\phi$ is
continuous, strictly increasing in $x_j$, and tends to $+\infty$ as
$x_j\to +\infty$, we conclude that $\phi(x_j,\bsy)
> 0$ for all $x_j\in\bbR$, and thus we integrate $x_j$ from $-\infty$ to
$\infty$. On the other hand, if $\bsy\in U_j$, in which case
$\phi(x_j,\bsy)$ changes sign once as $x_j$ goes from $-\infty$ to
$\infty$, then there exists a unique $x_j^* = \psi(\bsy)\in\bbR$ for which
$\phi(x_j^*,\bsy)=0$, and in this case we integrate $x_j$ from
$\psi(\bsy)$ to $\infty$. Hence we can write \eqref{lim} as
\begin{align*}
  (P_jf)(\bsy)
  &\,=\,
  \begin{cases}
  \displaystyle\int_{-\infty}^\infty \theta(x_j,\bsy)\,\rho(x_j)\,\rd x_j
  & \mbox{if } \bsy \in U_j^+, \vspace{0.1cm}\\
  \displaystyle\int_{\psi(\bsy)}^\infty \theta(x_j,\bsy)\,\rho(x_j)\,\rd x_j
  & \mbox{if } \bsy \in U_j.
  \end{cases}
\end{align*}
Note that $P_jf$ is continuous across the boundary between $U_j$ and
$U_j^+$, since from Lemma~\ref{lem:toinf} it follows that $\psi(\bsy)\to
-\infty$ as $\bsy\in U_j$ approaches a boundary point of $U_j$.

By the Leibniz Theorem and the Inheritance Theorem, we know that the
function $(P_j\theta)(\bsy) = \int_{-\infty}^\infty
\theta(x_j,\bsy)\,\rho(x_j)\,\rd x_j$ for $\bsy\in \bbR^{d-1}$ is as
smooth as~$\theta$, i.e., $P_j\theta \in \calW^r_{d-1,p,\bszeta}\cap
\calC^r(\bbR^{d-1})$. Therefore, to obtain the same smoothness property
for $P_jf$ it suffices that we consider in the remainder of this proof the
difference
\begin{equation}
 g(\bsy)
 \,:=\, (P_jf)(\bsy) - (P_j\theta)(\bsy)
 \,=\,  \begin{cases}
         - \displaystyle\int_{-\infty}^{\psi(\bsy)}  \theta(x_j,\bsy)\,\rho(x_j)\,\rd x_j & \mbox{if }  \bsy \in U_j,\\
         0 & \mbox{if } \bsy \in U^+_j.
        \end{cases}
\end{equation}
%

First we differentiate $g$ with respect to the $k$th coordinate for any
$k\ne j$.
For $\bsy\in U_j$ we obtain, using the fundamental theorem of calculus,
\begin{align} \label{dkpjf}
  (D_k g)(\bsy)
  &\,=\,
  -\int_{-\infty}^{\psi(\bsy)}
  (D_k\theta)(x_j,\bsy)\,\rho(x_j)\,\rd x_j
  -\, \theta(\psi(\bsy),\bsy) \,
  \rho(\psi(\bsy))\,
  (D_k\psi)(\bsy) \nonumber \\
  &\,=\,
  -\int_{-\infty}^{\psi(\bsy)}
  (D_k\theta)(x_j,\bsy)\,\rho(x_j)\,\rd x_j
  +\, \theta(\psi(\bsy),\bsy)\,
  \frac{(D_k \phi)(\psi(\bsy),\bsy)}{(D_j \phi)(\psi(\bsy),\bsy)}
 \,\rho(\psi(\bsy)),
\end{align}
where we substituted using \eqref{dkpsi-first}
\begin{equation} \label{dkpsi}
 (D_k \psi)(\bsy)
 \,=\, - \frac
 {(D_k \phi)(\psi(\bsy),\bsy)}
 {(D_j \phi)(\psi(\bsy),\bsy)}.
\end{equation}
It follows from \eqref{h1} and Lemma~\ref{lem:toinf} that both terms in
\eqref{dkpjf} go to $0$ as $\bsy\in U_j$ approaches a boundary point
$\bsy^\star$ of $U_j$ lying in $U_j^+$. Hence the condition
\eqref{Crprop:Assumption} in Lemma~\ref{lem:Crproperty} holds with $r=1$,
and we conclude that $g\in\calC^1(\bbR^{d-1})$.

Next we differentiate with respect to the $\ell$th coordinate for any
$\ell\ne j$ (allowing the possibility that $\ell=k$).
For $\bsy\in U_j$ it is useful to note that for any sufficiently smooth
$d$-variate function $\xi$ the rule for partial differentiation and the
chain rule gives
\begin{equation}\label{chainrule}
  D_\ell (\xi(\psi(\bsy),\bsy))
  \,=\, (D_\ell \xi )(\psi(\bsy),\bsy) + (D_j \xi )(\psi(\bsy),\bsy)\, (D_\ell \psi)(\bsy).
\end{equation}
Thus we find for $\bsy\in U_j$
\begin{align} \label{dldkpjf}
  &(D_\ell D_k g)(\bsy) \,=\,
  -\int_{-\infty}^{\psi(\bsy)}
  (D_\ell D_k\theta)(x_j,\bsy)\,\rho(x_j)\,\rd x_j
  \,-\, (D_k\theta)(\psi(\bsy),\bsy)\, \rho(\psi(\bsy)) \, (D_\ell\psi)(\bsy) \nonumber \\
  &\,+\, [(D_\ell\, \theta)(\psi(\bsy),\bsy) + (D_j\theta)(\psi(\bsy),\bsy)\,(D_\ell\psi)(\bsy)]\,
  \frac{(D_k \phi)(\psi(\bsy),\bsy)}{(D_j \phi)(\psi(\bsy),\bsy)}\,\rho(\psi(\bsy)) \nonumber \\
  &\,+\, \theta(\psi(\bsy),\bsy)\,
  \frac{[(D_\ell D_k \phi)(\psi(\bsy),\bsy) + (D_j D_k \phi)(\psi(\bsy),\bsy)\,(D_\ell\psi)(\bsy)]}
       {(D_j \phi)(\psi(\bsy),\bsy)}\,\rho(\psi(\bsy)) \nonumber \\
  &\,-\, \theta(\psi(\bsy),\bsy)\,
  \frac{(D_k \phi)(\psi(\bsy),\bsy)\,
  [(D_\ell D_j \phi)(\psi(\bsy),\bsy) + (D_j D_j \phi)(\psi(\bsy),\bsy)\,(D_\ell\psi)(\bsy)]}
  {[(D_j \phi)(\psi(\bsy),\bsy)]^2}\,\rho(\psi(\bsy))
  \nonumber \\
  &\,+\, \theta(\psi(\bsy),\bsy)\,
  \frac{(D_k \phi)(\psi(\bsy),\bsy)}{(D_j \phi)(\psi(\bsy),\bsy)}\,\rho'(\psi(\bsy))\,
  (D_\ell\psi)(\bsy),
\end{align}
where we used again \eqref{dkpsi}. We have from \eqref{h1} and
Lemma~\ref{lem:toinf} that all terms in \eqref{dldkpjf} go to $0$ as
$\bsy\in U_j$ approaches a boundary point $\bsy^\star$ of $U_j$ lying in
$U_j^+$. Hence the condition \eqref{Crprop:Assumption} in
Lemma~\ref{lem:Crproperty} holds with $r=2$, and we conclude that
$g\in\calC^2(\bbR^{d-1})$.

In general, for every non-zero multi-index
$\bsalpha=(\alpha_1,\ldots,\alpha_d)$ with $|\bsalpha|\le r$ and $\alpha_j
= 0$,
we claim that for $\bsy\in U_j$
\begin{equation} \label{nice}
  (D^\bsalpha g)(\bsy) \,=\,
  -\int_{-\infty}^{\psi(\bsy)}
  (D^\bsalpha\,\theta)(x_j,\bsy)\,\rho(x_j)\,\rd x_j
  \,+\, \sum_{m=1}^{M_{|\bsalpha|}} h_{\bsalpha,m}(\bsy),
\end{equation}
where $M_{|\bsalpha|}$ is a nonnegative integer, and each function
$h_{\bsalpha,m}$ is of the form \eqref{h}, with integers $\beta,a,b,c$ and
multi-indices $\bsgamma^{(i)}$ and $\bseta$ satisfying
\begin{equation} \label{bounds}
 1\le a,b\le 2|\bsalpha|-1,\quad
 1\le |\bsgamma^{(i)}|\le |\bsalpha|, \quad
 0\le |\bseta|,c\le |\bsalpha|-1, \quad
 1\le |\bsgamma^{(i)}|+|\bseta|+c \le |\bsalpha|.
\end{equation}
We have from \eqref{h1} and Lemma~\ref{lem:toinf} that all terms in
\eqref{nice} go to $0$ as $\bsy\in U_j$ approaches a boundary point
$\bsy^\star$ of $U_j$ lying in $U_j^+$. Hence the condition
\eqref{Crprop:Assumption} in Lemma~\ref{lem:Crproperty} holds for a
general $r$, and we conclude that $g\in\calC^r(\bbR^{d-1})$.

We will prove \eqref{nice}--\eqref{bounds} by induction on $|\bsalpha|$.
The case $|\bsalpha|=1$ is shown in \eqref{dkpjf}; there we have $M_1= 1$,
$a=1$, $b=1$, $c=0$, $\beta=1$, $|\bsgamma^{(1)}|=1$, $|\bseta| = 0$, and
$|\bsgamma^{(i)}|+|\bseta|+c =1$. The case $|\bsalpha|=2$ is shown in
\eqref{dldkpjf}; there we have $M_2= 8$, $1\le a,b\le 3$, $0\le c\le 1$,
$\beta = \pm 1$, $1\le |\bsgamma^{(i)}|\le 2$, $0\le |\bseta|\le 1$, and
$1\le |\bsgamma^{(i)}|+|\bseta|+c \le 2$. To establish the inductive step
we now differentiate $D^\bsalpha g$ once more: for $\ell\ne j$ we have
from~\eqref{nice}
\begin{align} \label{term}
  (D_\ell  D^\bsalpha g)(\bsy)
  &\,=\, -\int_{-\infty}^{\psi(\bsy)}
  (D_\ell  D^\bsalpha\theta)(x_j,\bsy)\,\rho(x_j)\,\rd x_j
   -\, (D^\bsalpha\theta)(\psi(\bsy),\bsy)\,
   \rho(\psi(\bsy)) \, (D_\ell  \psi)(\bsy) \nonumber \\
  &\qquad +\, \sum_{m=1}^{M_{|\bsalpha|}} (D_\ell\, h_{\bsalpha,m})(\bsy).
\end{align}
For a typical term in \eqref{term}, we have from \eqref{h}
\begin{align*}
  &(D_\ell\, h)(\bsy) \\
  &\,=\, \beta\,
  \frac{[(D_\ell D^\bseta \theta)(\psi(\bsy),\bsy) + (D_jD^\bseta \theta)(\psi(\bsy),\bsy)\,(D_\ell\psi)(\bsy)]
          \prod_{i=1}^a [(D^{\bsgamma^{(i)}}\phi)(\psi(\bsy),\bsy)]}
       {[(D_j\phi)(\psi(\bsy),\bsy)]^b}\, \rho^{(c)}(\psi(\bsy)) \\
  &\qquad + \beta\,
  \frac{(D^\bseta \theta)(\psi(\bsy),\bsy)\,D_\ell \big(\prod_{i=1}^a [(D^{\bsgamma^{(i)}}\phi)(\psi(\bsy),\bsy)]\big)}
       {[(D_j\phi)(\psi(\bsy),\bsy)]^b}\, \rho^{(c)}(\psi(\bsy)) \\
  &\qquad + \beta\,
  \frac{(D^\bseta \theta)(\psi(\bsy),\bsy)\,\prod_{i=1}^a [(D^{\bsgamma^{(i)}}\phi)(\psi(\bsy),\bsy)]}
       {[(D_j\phi)(\psi(\bsy),\bsy)]^b}\, \rho^{(c+1)}(\psi(\bsy))\,
  (D_\ell \psi)(\bsy) \nonumber\\
  &\qquad - \beta b\,
  \frac{(D^\bseta \theta)(\psi(\bsy),\bsy)\,\prod_{i=1}^a [(D^{\bsgamma^{(i)}}\phi)(\psi(\bsy),\bsy)]}
 {[(D_j\phi)(\psi(\bsy),\bsy)]^{b+1}} \,\rho^{(c)}(\psi(\bsy)) \nonumber\\
 &\qquad\qquad \Big[
  (D_\ell  D_j\phi)(\psi(\bsy),\bsy) + (D_jD_j\phi)(\psi(\bsy),\bsy)\,
  (D_\ell \psi)(\bsy) \Big], \nonumber
\end{align*}
where
\begin{align*}
  &D_\ell \bigg(\prod_{i=1}^a [(D^{\bsgamma^{(i)}}\phi)(\psi(\bsy),\bsy)] \bigg) \\
  &\,=\, \sum_{t=1}^a \Bigg(
  \Big[(D_\ell D^{\bsgamma^{(t)}}\phi)(\psi(\bsy),\bsy) + (D_j D^{\bsgamma^{(t)}}\phi)
  (\psi(\bsy),\bsy) \, (D_\ell \psi)(\bsy) \Big] 
  \prod_{\satop{i=1}{i\ne t}}^a (D^{\bsgamma^{(i)}}\phi)(\psi(\bsy),\bsy)
  \Bigg).
\end{align*}
Thus we conclude that $D_\ell\, h$ is a sum of functions of the form
\eqref{h}, but with $a$ and $b$ increased by at most $2$, $c$ increased by
at most $1$, $|\beta|$ multiplied by a factor of at most $b$,
$|\bsgamma^{(i)}|$ and $|\bseta|$ increased by at most $1$, and with
$|\bsgamma^{(i)}| + |\bseta| + c$ increased by at most $1$. Hence, $D_\ell
D^\bsalpha g$ consists of a sum of functions of the form \eqref{h}
satisfying the constraints in \eqref{bounds}. This completes the induction
proof for \eqref{nice}--\eqref{bounds}.

We now turn to the task of showing that $D^\bsalpha g\in
\calL_{p,\zeta_{d-1,\bsalpha_{-j}}}$ for $p\in[1,\infty)$ and all
$\bsalpha$ satisfying $|\bsalpha|\le r$ and $\alpha_j = 0$. We need to
consider
\begin{align*}
  \int_{\bbR^{d-1}}
  |(D^\bsalpha g)(\bsy)|^p\,
  \zeta_{d-1,\bsalpha_{-j}}(\bsy)\,
  \rd\bsy 
  &\,=\,
  \int_{U_j}  |(D^\bsalpha g)(\bsy)|^p\,
  \zeta_{d-1,\bsalpha_{-j}}(\bsy)\,
  \rd\bsy,
\end{align*}
where we have split the integral noting that $U_j$ is open and its
complement $U_j^+$ is closed, as they are both Borel measurable, and that
$D^\bsalpha g$ is zero on $U_j^+$.


For $\bsy\in U_j$, it follows from H\"older's inequality and the special
form of $D^\bsalpha g$ in \eqref{nice} that
\begin{align*}
  |(D^\bsalpha g)(\bsy)|^p
  &\,=\, \left| - \int_{-\infty}^{\psi(\bsy)}
  (D^\bsalpha \theta)(x_j,\bsy)\,\rho(x_j)\, \rd x_j
  + \sum_{m=1}^{M_{|\bsalpha|}} h_{\bsalpha,m}(\bsy)\right|^p \\
  &\,\le\, \left( \int_{-\infty}^{\psi(\bsy)}
  |(D^\bsalpha \theta)(x_j,\bsy)|\, \rho(x_j)\, \rd x_j
  + \sum_{m=1}^{M_{|\bsalpha|}} |h_{\bsalpha,m}(\bsy)|\right)^p \\
  &\,\le\, (M_{|\bsalpha|}+1)^{p-1} \left( \left(\int_{-\infty}^{\psi(\bsy)}
  |D^\bsalpha \theta(x_j,\bsy)|\, \rho(x_j)\, \rd x_j\right)^p
  + \sum_{m=1}^{M_{|\bsalpha|}} |h_{\bsalpha,m}(\bsy)|^p\right) \\
  &\,\le\, (M_{|\bsalpha|}+1)^{p-1} \left(\int_{-\infty}^{\psi(\bsy)}
  |(D^\bsalpha \theta)(x_j,\bsy)|^p\, \rho(x_j)\, \rd x_j
  + \sum_{m=1}^{M_{|\bsalpha|}} |h_{\bsalpha,m}(\bsy)|^p\right),
\end{align*}
and thus
\begin{align*}
  & \int_{U_j} |(D^\bsalpha g)(\bsy)|^p\,
  \zeta_{d-1,\bsalpha_{-j}}(\bsy) \,\rd\bsy\\
  &\,\le\,
  (M_{|\bsalpha|}+1)^{p-1} \Bigg(\int_{\bbR^d}
  |(D^\bsalpha \theta)(\bsx)|^p\,\zeta_{d,\bsalpha}(\bsx)\,\rd\bsx
  + \sum_{m=1}^{M_{|\bsalpha|}}
  \int_{U_j}  |h_{\bsalpha,m}(\bsy)|^p \,\zeta_{d-1,\bsalpha_{-j}}(\bsy)\,\rd\bsy
  \Bigg)
  \,<\, \infty,
\end{align*}
with the finiteness coming because $\rho(x_j) = \zeta_0(x_j) =
\zeta_{\alpha_j}(x_j)$ and $\theta\in\calW^r_{d,p,\bszeta}$, and because
each integral involving $h_{\bsalpha,m}$ is finite due to the
condition~\eqref{h2}.
This 
proves that $\is{g} \in \calW^r_{d-1,p,\bszeta}$ as claimed.
\end{proofof}

\textbf {Acknowledgements} The authors acknowledge the support of the
Australian Research Council under the projects FT130100655 and
DP150101770.

\vskip 3pc


\begin{itemize}

\item[]
{\sc Andreas Griewank} \\
School of Mathematical Sciences and Information Technology\\
Yachay Tech\\
Urcuqui 100119, Imbabura\\
Ecuador\\ email: griewank@yachaytech.edu.ec

\item[]
{\sc Frances Y. Kuo} \\
School of Mathematics and Statistics \\
The University of New South Wales \\
Sydney NSW 2052, Australia \\
email: f.kuo@unsw.edu.au

\item[]
{\sc Hernan Le\"ovey} \\
Structured Energy Management Team\\
Axpo AG \\
Baden, Switzerland \\
Post: Parkstrasse 23 CH-5401 Baden, Switzerland \\
email: hernaneugenio.leoevey@axpo.com

\item[]
{\sc Ian H. Sloan} \\
School of Mathematics and Statistics \\
The University of New South Wales \\
Sydney NSW 2052, Australia \\
email: i.sloan@unsw.edu.au

\end{itemize}


\begin{thebibliography}{99}

\setlength{\parsep }{-0.5ex}
\setlength{\itemsep}{-0.5ex}
\newcommand\BAMS{\emph{Bull. Amer. Math. Soc.\ }}
\newcommand\BIT{\emph{BIT\ }}
\newcommand\Com{\emph{Computing\ }}
\newcommand\CA{\emph{Constr. Approx.\ }}
\newcommand\FCM{\emph{Found. Comput. Math.\ }}
\newcommand\JAT{\emph{J. Approx. Th.\ }}
\newcommand\JC{\emph{J. Complexity\ }}
\newcommand\JCP{\emph{J. of Computational Physics\ }}
\newcommand\JMA{\emph{SIAM J. Math. Anal.\ }}
\newcommand\JMAA{\emph{J. Math. Anal. Appl.\ }}
\newcommand\JMM{\emph{J. Math. Mech.\ }}
\newcommand\JMP{\emph{J. Math. Physics\ }}
\newcommand\MC{\emph{Math. Comp.\ }}
\newcommand\NA{\emph{Numer. Alg.\ }}
\newcommand\NM{\emph{Numer. Math.\ }}
\newcommand\RMJ{\emph{Rocky Mt. J. Math.\ }}
\newcommand\SJNA{\emph{SIAM J. Numer. Anal.\ }}
\newcommand\SR{\emph{SIAM Rev.\ }}
\newcommand\TAMS{\emph{Trans. Amer. Math. Soc.\ }}
\newcommand\TCS{\emph{Theoretical Computer Science\ }}
\newcommand\TOMS{\emph{ACM Trans. Math. Software\ }}
\newcommand\USSR{\emph{USSR Comput. Maths. Math. Phys.\ }}
\frenchspacing

\bibitem{ACN13a}
 N.~Achtsis, R.~Cools, and D.~Nuyens,
 \emph{Conditional sampling for barrier option pricing under the LT method},
 SIAM J.\ Financial Math. \textbf{4} (2013), 327--352.

\bibitem{ACN13b}
 N.~Achtsis, R.~Cools, and D.~Nuyens,
 \emph{Conditional sampling for barrier option pricing under the Heston
 model}, in: J.~Dick, F.Y. Kuo, G.W. Peters, I.H. Sloan (eds.), {M}onte {C}arlo
 and Quasi-{M}onte {C}arlo Methods 2012, pp. 253--269, Springer-Verlag,
 Berlin/Heidelberg (2013).

\bibitem{BST17}
 C.~Bayer, M.~Siebenmorgen, and R.~Tempone,
 \emph{Smoothing the payoff for efficient computation of Basket option
 prices}, Quantitive Finance, published online 20 July 2017, 1--15.

\bibitem{BG04}
 H.~Bungartz and M.~Griebel,
 \emph{Sparse grids}, Acta Numer.\ \textbf{13} (2004), 147--269.

\bibitem{DKS13}
  J.~Dick, F.~Y. Kuo, and I.~H. Sloan,
  \emph{High dimensional integration -- the quasi-Monte Carlo way},
  Acta Numer.\ 22 (2013), 133--288.

\bibitem{Glasserman} P.~Glasserman, Monte Carlo Methods in Financial
    Engineering, Springer-Verlag, 2003.

\bibitem{GlaSta01}
 P.~Glasserman and J.~Staum, \emph{Conditioning on one-step survival for
 barrier option simulations}, Oper.\ Res., \textbf{49} (2001), 923--937.

\bibitem{GKS10} M.~Griebel, F.~Y.~Kuo, and I.~H.~Sloan, \emph{The
    smoothing effect of the ANOVA decomposition},
    J.\ Complexity\ \textbf{26} (2010), 523--551.

\bibitem{GKS13} M.~Griebel, F.~Y.~Kuo, and I.~H.~Sloan, \emph{The
    smoothing effect of integration in $\bbR^d$ and the ANOVA decomposition},
    Math.\ Comp.\ \textbf{82} (2013), 383--400.

\bibitem{GKSnote} M.~Griebel, F.~Y. Kuo, and I.~H. Sloan, \emph{Note on
    ``the smoothing effect of integration in~$\bbR^d$ and the ANOVA
    decomposition''}, Math.\ Comp.\ \textbf{86} (2017), 1847--1854.

\bibitem{Hol11}
  M.~Holtz,
  Sparse Grid Quadrature in High Dimensions with Applications in Finance and Insurance
  (PhD thesis), Springer-Verlag, Berlin, 2011.

\bibitem{JoeKuo08}
 S.~Joe and F.~Y.~Kuo,
 \emph{Constructing Sobol$'$ sequences with better two-dimensional projection},
 SIAM J.\ Sci.\ Comput.\ \textbf{30} (2008), 2635--2654.

\bibitem{KSWW10} F.~Y.~Kuo, I.~H.~Sloan, G.~W.~Wasilkowski and
    H.~W\'ozniakoski, \emph{On decompositions of multivariate functions},
    Math.\ Comp., \textbf{79} (2010), 953--966.

\bibitem{Matousek}
 J.~Matousek,
 \emph{On the L2-discrepancy for anchored boxes},
  J.\ Complexity \textbf{14} (1998), 527--556.


\bibitem{NuyWat12}
 D.~Nuyens and B.~J.~Waterhouse,
 \emph{A global adaptive quasi-Monte Carlo algorithm for functions of low truncation dimension
 applied to problems from finance}, in: L.~Plaskota and H.~Wo\'zniakowski (eds.), {M}onte {C}arlo
 and Quasi-{M}onte {C}arlo Methods 2010, pp. 589--607, Springer-Verlag,
 Berlin/Heidelberg (2012).


\bibitem{Owen05} A. B. Owen, \emph{Multidimensional variation for
    quasi-{Monte Carlo}}, International Conference on Statistics in honour
    of Professor Kai-Tai Fang's 65th birthday. Jianqing Fan and Gang Li
    editors, 49--74 (2005).

\bibitem{Sche07} K. Scheicher, Complexity and effective dimension of
    discrete L\'evy areas, \emph{J. Complexity} \textbf{23} (2007),
    152--168.

\bibitem{Sobol90} I.~M.~Sobol$'$, \emph{Sensitivity estimates for
    nonlinear mathematical models}, Matematicheskoe Modeliravonic, 1990, V. 2, N. 1,
    112--118 (in Russian).  English translation in Mathematical Modeling
    and Computatinoal Experiment, 407--414 (1993).



\bibitem{WWH17}
 C.~Weng, X.~Wang, and Z.~He,
 \emph{Efficient computation of option prices and greeks by quasi-Monte Carlo
 method with smoothing and dimension reduction},
 SIAM J.\ Sci.\ Comput.\ \textbf{39} (2017), B298--B322.

\end{thebibliography}
\end{document}